\documentclass[12pt]{amsart}

\usepackage{amsmath,amssymb,color}
\usepackage{amsthm}
\usepackage{hyperref}
\usepackage[margin=1.0in]{geometry}
\usepackage{mathtools,slashed}

\numberwithin{equation}{section}

\newtheorem{prop}{Proposition}
\newtheorem{lemma}[prop]{Lemma}

\newtheorem{thm}[prop]{Theorem}
\newtheorem{cor}[prop]{Corollary}

\numberwithin{prop}{section}

\theoremstyle{definition}
\newtheorem{defn}[prop]{Definition}

\newtheorem{rmk}[prop]{Remark}

\newcommand{\gL}{\Lambda}

\newcommand{\N}{\nabla}

\DeclareMathOperator{\Rm}{Rm}
\DeclareMathOperator{\inj}{inj}
\DeclareMathOperator{\tr}{tr}

\DeclareMathOperator{\diam}{diam}

\DeclareMathOperator{\End}{End}

\usepackage{mathrsfs}

\newcommand{\ttau}{{\tilde{\tau}}}
\newcommand{\tnabla}{{\tilde{\nabla}}}
\newcommand{\lot}{{\rm L.O.T.}}
\newcommand{\rG}{{\rm G}}
\newcommand{\SU}{{\rm SU}}
\DeclareMathOperator\vol{vol}

\newcommand{\qandq}{\quad\text{and}\quad}
\newcommand{\qwhereq}{\quad\text{where}\quad}

\newcommand{\qforq}{\quad\text{for}\quad}
\newcommand{\sX}{\mathscr{X}}
\def\bR{\mathbb R}

\newcommand{\fo}{{\mathfrak o}}
\newcommand{\fg}{{\mathfrak g}}
\newcommand{\fs}{{\mathfrak s}}
\newcommand{\GL}{\mathrm{GL}}
\def\fso{\mathfrak{so}}
\newcommand{\sym}{\mathrm{sym}}

\newcommand{\bS}{\mathbb{S}}
\newcommand{\Ric}{{\rm Ric}}
\def\cL{\mathcal{L}}
\DeclareMathOperator\curl{curl}
\DeclareMathOperator\Div{div}
\newcommand{\tT}{\widetilde{T}}
\newcommand{\tvarphi}{\widetilde{\varphi}}
\newcommand{\tpsi}{\widetilde{\psi}}

\newcommand{\tg}{\widetilde{g}}
\allowdisplaybreaks
\begin{document}

\title[A parabolic flow for the large volume heterotic $\rG_2$ system]{A parabolic flow for the large volume \\heterotic $\rG_2$ system}

\author{Mario Garcia-Fernandez}
\address{Instituto de Ciencias Matem\'aticas (CSIC-UAM-UC3M-UCM)\\ Nicol\'as Cabrera 13--15, Cantoblanco\\ 28049 Madrid, Spain}
\email{\href{mario.garcia@icmat.es}{mario.garcia@icmat.es}}

\author{Andres J. Moreno}
 \address{Institute of Mathematics, Statistics and Scientific Computing (IMECC), University of Campinas (Unicamp), 13083-859 Campinas-SP, Brazil.}
 \email{\href{mailto:}{amoreno@unicamp.br}}

\author{Alec Payne}
\address{North Carolina State University, 2311 Stinson Drive, Raleigh, NC 27607, USA}
\email{\href{mailto:ajpayne4@ncsu.edu}{ajpayne4@ncsu.edu}}

 \author{Jeffrey Streets}
 \address{Rowland Hall\\
          University of California, Irvine\\
          Irvine, CA 92617, USA}
 \email{\href{mailto:jstreets@uci.edu}{jstreets@uci.edu}}

\date{\today}

\thanks{MGF and AM were partially supported by the Spanish Ministry of Science and Innovation, through the `Severo Ochoa Programme for Centres of Excellence in R\&D' (CEX2023-001347-S), and under grants PID2022-141387NB-C22 and CNS2022-135784. JS was supported by the NSF via DMS-2342135. AM has also been funded by the São Paulo Research Foundation (Fapesp)  [2021/08026-5] and [2023/13780-6].}

\begin{abstract}
    We introduce a geometric flow of conformally coclosed $\rG_2$-structures, whose fixed points are large volume solutions of the heterotic $\rG_2$ system, with vanishing scalar torsion class $\tau_0 = 0$. After conformal rescaling, it becomes a flow of coclosed $\rG_2$-structures, related to Grigorian's modified $\rG_2$ coflow, which is coupled to a flow for a \emph{dilaton function}. Our main results establish fundamental short-time existence and Shi-type smoothing properties of this flow, as well as a classification of its fixed points. By a classical rigidity result in the string theory literature, the fixed points on a compact manifold correspond to torsion-free $\rG_2$-structures, that is, to metrics with holonomy contained in $\rG_2$. Thus, we establish in the affirmative a folklore question in the special holonomy community, about the existence of a well-posed flow for coclosed $\rG_2$-structures with fixed points given by torsion-free $\rG_2$-structures. The flow also satisfies a monotonicity formula for the \emph{$\rG_2$-dilaton functional} (\emph{volume scale} in string theory), which allows us to strengthen the rigidity result with an alternative proof.  The monotonicity of the $\rG_2$-dilaton functional, combined with the Shi-type estimates, leads to a general result on the convergence of nonsingular solutions. A dimension reduction analysis reveals an interesting link with natural flows for $\SU(3)$-structures, previously introduced in the literature.

\end{abstract}

\maketitle

\section{Introduction}

To date, several geometric flows have been proposed to construct torsion-free $\rG_2$-structures. Bryant proposed the Laplacian flow for closed $\rG_2$-structures~\cite{Bryant2006, bryant2011laplacian}, which arises naturally as the gradient flow for the Hitchin volume functional \cite{hitchin2000geometry}. Later, Lotay--Wei \cite{LotayWeiLaplacianFlow} showed the dynamic stability of this flow. The Laplacian coflow for coclosed $\rG_2$-structures was proposed by Karigiannis, McKay, and Tsui \cite{karigiannis2012soliton}. While formally appealing, this flow lacks a general short-time existence theory. In light of this, Grigorian introduced a modified Laplacian coflow which enjoys a short-time existence theory, although there is no characterization of the fixed points.  Many results on these and other flows of $\rG_2$-structures have appeared in recent years (e.g., see  \cite{bedulli2020G2,GaoChenShi,collins2022spinorflowsfluxi,dwivedi2021gradient,DwivediGianniotisKarigiannis, fino2023twisted,FinoFowdar,KarigiannisFlowsI,LotayStein}).

In this work, we propose a novel geometric flow of $\rG_2$-structures inspired by recent progress on generalized Ricci solitons and by work in the string theory literature by Ashmore, Minasian, and Proto \cite{Ashmore:2023ift}.  A classic result of Friedrich--Ivanov \cite{friedrich2003parallel_string_theory} shows that a $\rG_2$-structure admits a compatible connection with skew-symmetric torsion if and only if the torsion component $\tau_2 = 0$, whence the $\rG_2$-structure is called integrable.  The resulting connection $\N^+$ is in fact unique with torsion
\begin{align*}
    H_\varphi = \frac{1}{6}\tau_0\varphi - \tau_1\lrcorner\psi - \tau_3.
\end{align*}
If the characteristic torsion $H_{\varphi}$ is \emph{closed}, then the connection with opposite torsion $\N^-$ is a $\rG_2$-instanton and, by recent work of Ivanov-Stanchev \cite{ivanov2025}, the pair $(g_{\varphi}, H_{\varphi})$ defines a generalized Ricci soliton, and is a solution to the heterotic $\rG_2$ system, in the sense of \cite{SGFLS}.

In light of this, it is natural to deform the four-form $\psi$ by $dH_{\varphi}$ within the class of coclosed $\rG_2$-structures to construct solutions to the heterotic $\rG_2$ system. To obtain a meaningful flow, however, this initial proposal must be modified in two significant ways. First, the operator $dH_{\varphi}$ fails to be elliptic in much the same way as the operator $d d^\star_{\varphi} \psi$, thus we require a modification similar to Grigorian \cite{Grigorian2013}.  Furthermore, we flow within the class of \emph{conformally} coclosed $\rG_2$-structures. This is natural from the point of view of string theory, and thus we couple the flow of the $\rG_2$-structure to a parabolic flow of conformal factors. To wit, given real parameters $\gamma,\sigma \in \bR$, we say (cf.\ Definition \ref{d:flowdef}) that a one-parameter family of $\rG_2$-structures $\varphi_t \in \Omega^3_+$ and smooth dilaton functions $\phi_t \in C^\infty(M,\bR)$ satisfy the \emph{heterotic $\rG_2$ flow} if
\begin{gather}\label{eq:heteroticG2flowintro}
\begin{split}
    \frac{\partial}{\partial t}\left(e^{-4\phi}\psi\right)=&-dH_{\varphi}-\frac{7}{3}d\left(\tau_0\varphi \right),\\
    \frac{\partial}{\partial t}\phi= &\ e^{4\phi} \left(\Delta \phi - \gamma |d\phi|^2 + \frac{1}{4}|\tau_3|^2 - \sigma \tau_0^2 \right).
\end{split}
\end{gather}
A key feature of this flow is that it preserves the conformally coclosed condition $d \left(e^{-4\phi}\psi\right) = 0$. Notice that in particular, a flow line with conformally coclosed initial data will remain integrable along the flow, i.e.\ $\tau_2 = 0$ will hold for all time.

For specific values of $\sigma,\gamma > 0$ (see \eqref{eq:sigmagammahet}), Proposition \ref{prop:weakG2} gives a complete characterization of the fixed points of \eqref{eq:heteroticG2flowintro}: they correspond to solutions of the coupled equations
\begin{equation}
\label{eq:heteroticG2intro}
\begin{split}
    d (e^{-4\phi}\psi) & = 0,\\
    dH_\varphi & = 0,\\
     \mathcal{S}^+_{\varphi} := R_g-\frac{1}{2}|H_\varphi|^2 - 8 \Delta \phi - \left|d\phi\right|^2 & = 0,
\end{split}
\end{equation}
where $g = g_\varphi$ is the Riemannian metric determined by $\varphi$. This system of equations corresponds formally to the heterotic $\rG_2$-system in the \emph{large volume limit} $\alpha' \to 0$ \cite{de_la_Ossa_2016_instanton_bundles} (cf.\ \cite{Clarke2016,Clarke2022,SGFLS}), once we use the classical ansatz given by the identification of the spin connection with skew-torsion $-H_\varphi$ and the gauge $\rG_2$-instanton connection (also known as \emph{standard embedding solutions} \cite{CHSW}). Solutions of \eqref{eq:heteroticG2intro} determine \emph{gradient generalized Ricci solitons}, i.e.\ $(g_\varphi,H_\varphi, \phi)$ solve
\begin{align*}
    \Ric_g-\frac{1}{4}H^2_{\varphi}+2\cL_{\nabla \phi}g & =0,\\
    d^{\star}H_\varphi + 4 \nabla \phi \lrcorner H_{\varphi} & = 0.
\end{align*}
Due to a classical rigidity result in the string theory literature by Nu\~nez and Maldacena \cite{Maldacena:2000mw}, the further vanishing of the scalar curvature $\mathcal{S}^+_{\varphi}$ renders solutions of these equations on a compact manifold rigid, i.e.\ $H_{\varphi} = 0$ and $\phi$ is constant, and hence correspond to torsion-free $\rG_2$-structures.  Our sharp characterization of the fixed points for the flow \eqref{eq:heteroticG2flowintro}, in Proposition \ref{prop:weakG2}, contrasts with the lack of such a characterization for the flows recently proposed in~\cite{FinoFowdar,karigiannis2025flows}.

In fact, we may characterize the fixed points of ~\eqref{eq:heteroticG2flowintro} for a large range of $\sigma, \gamma < 0$ (see Proposition \ref{prop:weakG2functional}). For $\sigma, \gamma \ll 0$, the flow acquires nicer qualitative properties. In this case, the \emph{$\rG_2$-dilaton functional}, defined by
\begin{gather}\label{eq:dilatonfintro}
\mathcal{M}= \int_M e^{-4\phi} \operatorname{Vol}_\varphi,
\end{gather}
is monotone along the flow \eqref{eq:heteroticG2flowintro} and is constant if and only if the $\rG_2$-structure is torsion-free and the dilaton is constant (see Theorem \ref{t:MonotonicityIntro}). The functional ~\eqref{eq:dilatonfintro} is inspired by the variational approach to the six-dimensional heterotic system proposed in \cite{garciafern2018canonical}, better known as the Hull-Strominger system, and provides an analogue in seven dimensions of the so-called \emph{dilaton functional} for $\operatorname{SU}(3)$-structures. It has a physical interpretation as a \emph{volume scale} and has been previously considered in the string theory literature \cite{deIaOssa:2019cci}. 

To state our main results, we consider a more general flow than ~\eqref{eq:heteroticG2flowintro}, the \emph{generic heterotic $G_2$ flow}, which includes an additional parameter $C \in \mathbb{R}$:
\begin{gather}\label{eq: mod_anomaly_flowCbetaintro}
\begin{split}
    \frac{\partial}{\partial t}\left(e^{-4\phi}\psi\right)=&-dH_{\varphi}+\frac{7}{4}C d\left(\tau_0\varphi \right),\\
    \frac{\partial}{\partial t}\phi=&\ e^{4\phi} \left(\Delta_g \phi - \gamma |d\phi|^2 + \frac{1}{4}|\tau_3|^2 - \sigma \tau_0^2 \right).
\end{split}
\end{gather}

The generic heterotic $\rG_2$ flow~\eqref{eq: mod_anomaly_flowCbetaintro} becomes the heterotic $\rG_2$ flow \eqref{eq:heteroticG2flowintro} when setting $C = -4/3$. We will see that the choice of $C = -4/3$ is natural from an analytic perspective (Theorem \ref{t:maththm2intro}), but it is convenient to use the parameter $C$ in some of our results.

Our first main result establishes the general well-posedness for the generic heterotic $\rG_2$ flow for a large open set of parameters, on compact manifolds with conformally coclosed initial data. The well-posedness holds for the heterotic $\rG_2$ flow ~\eqref{eq:heteroticG2flowintro} in particular. 

\begin{thm}[cf.\ Theorem \ref{thm:STE}]\label{t:mainthm1}
Let $M$ be compact, and let $(\varphi_0,\phi_0) \in \Omega^3_+ \times C^\infty(M,\bR)$ such that $d(e^{-4\phi_0} \psi_0) = 0$. Then, for any $C<-\frac13$ and $\gamma,\sigma\in \bR$, there exists $\epsilon>0$, depending on $(\varphi_0, \phi_0)$, and a unique solution $(\varphi_t,\phi_t)$ of the generic heterotic $\rG_2$ flow \eqref{eq: mod_anomaly_flowCbetaintro} defined for $t \in [0,\epsilon)$, with initial condition $(\varphi_0, \phi_0)$, such that
$$
d(e^{-4\phi_t} \psi_t) = 0
$$ 
for all $t\in [0,\epsilon)$.
\end{thm}

\noindent We note that our flow can be equivalently expressed as a coupled flow of coclosed $\rG_2$-structures and dilaton (see Lemma \ref{lem:confCbeta}). By our characterization of the fixed points in Proposition \ref{prop:weakG2} and Proposition \ref{prop:weakG2functional}, we have thus established in the affirmative a folklore question in the special holonomy community, about the existence of a well-posed flow for coclosed $\rG_2$-structures with meaningful fixed points and a well-behaved short-time existence theory. 

We now state our monotonicity result for the $\rG_2$ dilaton functional along generic heterotic $\rG_2$ flow. By choosing $\gamma, \sigma$ sufficiently negative, e.g.\ $\gamma = \sigma = -2$, the $\rG_2$ dilaton functional is monotone along the heterotic $\rG_2$ flow ~\eqref{eq:heteroticG2flowintro}.

\begin{thm}[cf.\ Theorem \ref{thm: monotonicity funct}]\label{t:MonotonicityIntro}
      Let $(\varphi_t, \phi_t)$ be a solution to the generic heterotic $\rG_2$ flow ~\eqref{eq: mod_anomaly_flowCbetaintro} on a compact manifold $M$ with initial condition $(\varphi_0, \phi_0)$ satisfying $d(e^{-4\phi_0} \psi_0)=0$. Then, for all $t$,
$$
\frac{\partial}{\partial t}\mathcal{M}(\varphi_t,\phi_t) =\int_M \left(-(3\gamma + 3) |d\phi|^2 + |\tau_3|^2 + 3\left(\frac{7(21 C - 2)}{144} - \sigma\right)\tau_0^2 \right)\operatorname{Vol}_\varphi .
$$
Provided that 
$$
\gamma < -1, \qquad \sigma < 
 \frac{7(21 C - 2)}{144},
$$
one has $\frac{\partial}{\partial t}\mathcal{M}(\varphi_t,\phi_t) \geqslant 0$, with equality only if $\varphi_t$ is torsion-free and $\phi_t$ is constant.
\end{thm}
The monotonicity of the $\rG_2$ dilaton functional provides some conceptual justification for the usage of the Grigorian-type modification term $\frac{7}{4}C d(\tau_0 \varphi)$, justified in previous work only on analytic grounds. We note that Theorem \ref{t:MonotonicityIntro} gives a characterization of the fixed points of the generic heterotic $\rG_2$ flow under the assumptions on $\gamma, \sigma, C$ (see Proposition~\ref{prop:weakG2functional}). 

The next main result establishes smoothing estimates for the heterotic $\rG_2$ flow \eqref{eq:heteroticG2flowintro} under natural bounds. We should emphasize that these estimates are only valid for the flow \eqref{eq: mod_anomaly_flowCbeta} with the specific value of the parameter $C = -4/3$, which motivates our definition of the heterotic $\rG_2$ flow \eqref{eq:heteroticG2flowintro}. The choice of $C$ gets rid of a problematic second-order term in the evolution of the torsion (see Proposition \ref{proposition evolution nabla^k T} and ~\eqref{eqn:evolution T anomaly 2}).

\begin{thm}[cf.\ Theorem \ref{t:maththm2bulk}] \label{t:maththm2intro}
    Let $(\varphi_t, \phi_t)$ be a solution to the heterotic $\rG_2$ flow ~\eqref{eq:heteroticG2flowintro} with initial condition $(\varphi_0, \phi_0)$ satisfying $d(e^{-4\phi_0} \psi_0)=0$. Let $B_r(p)$ be a ball of radius $r$ around $p \in M$ with respect to the metric $g_{\varphi_0}$. Suppose that there exist $t_0, \Lambda>0$ such that
\begin{align*}
        \left|\Rm\right| + |T|^2 + |\nabla T| + |\phi| < \Lambda
\end{align*}
on $B_r(p) \times [0, t_0]$.  Then, for each $k \geq 0$, there is a constant $C(k, t_0, \Lambda,\gamma,\sigma, r)$ such that
\begin{align*}
    \left|\nabla^{k+2}\phi\right| + \left| \nabla^k \Rm\right| + \left|\nabla^{k+1} T\right| < C(k, t_0, \Lambda, \gamma,\sigma, r)
\end{align*}
on $B_{r/2}(p) \times [t_0/2, t_0]$.
\end{thm}
As a consequence of the analytic theory we develop, together with the monotonicity of the $\rG_2$-dilaton functional, we obtain a long-time existence and convergence result for certain nonsingular solutions:
\begin{cor} [cf. Corollary \ref{c:nonsingularbulk}] \label{c:nonsingularintro} 
    Let $(\varphi_t, \phi_t)$ be a solution to the heterotic $\rG_2$ flow ~\eqref{eq:heteroticG2flowintro} on a compact manifold $M$ with initial condition $(\varphi_0, \phi_0)$ satisfying $d(e^{-4\phi_0} \psi_0)=0$. Assume furthermore that there exists a constant $\Lambda > 0$ such that for all $t > 0$ for which the flow exists,
    \begin{align*}
        \max \{ \inj^{-1}_{g}, \diam(g), \left|\Rm\right|, |T|^2, |\nabla T|, |\phi| \} < \Lambda.
    \end{align*}
    Then the flow exists on $[0,\infty)$, and for any sequence $\{t_k\} \to \infty$ there exists a subsequence such that $(\varphi_{t_{k_j}}, \phi_{t_{k_j}})$ converges to $(\varphi_{\infty}, \phi_{\infty})$ where $\varphi_{\infty}$ is a torsion-free $\rG_2$-structure and $\phi_{\infty}$ is constant.
\end{cor}

In Section \ref{sec:dimred}, we analyze a circle symmetry reduction of the generic heterotic $\rG_2$ flow \eqref{eq: mod_anomaly_flowCbetaintro}, finding an interesting relation with the \emph{anomaly flow} of Phong, Picard, Zhang \cite{PPZ1,PPZ2} for conformally coclosed integrable $\operatorname{SU}(3)$-structures and the \emph{Type IIA flow} \cite{IIAflow}.  The nomenclature for these flows arises from the relationship to the \emph{Green-Schwarz anomaly cancellation mechanism} in string theory.  For instance, a circle-invariant family of $\rG_2$-structures
\begin{equation*}
\varphi = \omega\wedge\theta+\rho_+
\end{equation*}
satisfies the generic heterotic $\rG_2$ flow \eqref{eq: mod_anomaly_flowCbetaintro} if and only if the corresponding $\operatorname{SU}(3)$-structure $(\omega,\rho_+,\rho_-)$ and the six-dimensional dilaton $\phi$ satisfy the following evolution equations 
\begin{align}\nonumber
\partial_t \left(e^{-4\phi}\frac{\omega^2}{2}\right)=& \ dd^c\omega + d \hat{N} + 2Cd\left(\pi_0\rho_+\right),\\\label{eq:6dflowintro}
        \partial_t\left(e^{-4\phi}\rho_- \right)=& \ d\left(\omega\lrcorner d\rho_+\right)+2\left(C-\frac{2}{3}\right)d(\pi_0\omega)-8d(\nabla\phi\lrcorner\rho_-).\\ \nonumber
        \partial_t\phi=& \ e^{4\phi}\left(\Delta\phi+(4-\gamma)|d\phi|^2+\frac14|H_\omega|^2+\frac14|\pi_2|^2-\left(\frac{28}{63}+\sigma\right)\pi_0^2\right),
\end{align}
where $H_\omega = - d^c\omega - \hat{N}$, for $d^c\omega=-Jd\omega$ and $\hat{N}$ the totally skew-symmetric part of the Nijenhuis tensor, and the tensors $\pi_i$ are components of the Chiossi--Salamon decomposition of torsion \cite{chiossi2002}.  For simplicity, in this introduction we have stated the formula \eqref{eq:6dflowintro} in the case of a fixed trivial connection $\theta$; a general formula, including the variation of $\theta$, can be found in Proposition \ref{prop:S1reducedflow}. The previous coupled system of evolution equations provides a formal \emph{interpolation} between the aforementioned flows of $\operatorname{SU}(3)$-structures (see Remark \ref{rem:interpolation}). This suggests natural modifications of the anomaly flow and the Type IIA flow, where the dilaton has an independent evolution equation, which may prevent the formation of finite-time singularities (cf. \cite[Section 5.4]{Picard2024}). We also speculate on its salient implications in the context of $T$-duality and non-K\"ahler mirror symmetry (cf.\ \cite{Fei:2023}).

\medskip
\noindent\textbf{Acknowledgments:} The authors wish to thank  Mark Haskins, Jason Lotay, Carl Tipler, and Luigi Vezzoni for helpful conversations.  After the main results of this work were completed, we became aware of the overlapping work of Karigiannis, Picard, and Suan \cite{karigiannis2025flows}.

\section{Preliminaries in \texorpdfstring{$\rG_2$}{G2}-geometry}
\subsection{Linear \texorpdfstring{$\rG_2$}{G2}-geometry}

In this section, we recall some facts on $\rG_2$-structures and their torsion, mainly from \cite{Bryant2006, KarigiannisFlowsI}.  Let $M$ be an oriented, spin manifold of dimension seven. A  differential $3$-form $\varphi$ on $M$ is called a \emph{$\rG_2$-structure}, if for each $x\in M$ there is an invertible map  $u:T_xM\rightarrow \bR^7$, such that $u^*\varphi_0=\varphi_x$ where
\begin{equation*}
\varphi_0=e^{127}+e^{347}+e^{567}+e^{135}-e^{146}-e^{236}-e^{245},
\end{equation*}
with $e^{ijk}:=e^i\wedge e^j \wedge e^k$ and $\{e^1,...,e^7\}$ is the dual basis of the standard basis of $\bR^7$. The space of $\rG_2$-structures will be denoted by $\Omega^3_+$. The presence of a $\rG_2$-structure uniquely determines a Riemannian metric $g=g_\varphi$ and a volume form $\vol=\vol_\varphi = \vol_{g_{\varphi}}$ on $M$, by 
\begin{equation*}
    (X\lrcorner \varphi)\wedge (Y\lrcorner \varphi)\wedge\varphi=6g(X,Y)\vol, \qforq X,Y\in \sX(M).
\end{equation*}
Hence $\varphi$ induces the Levi-Civita connection $\nabla$, a Hodge star operator $\star_\varphi=\star$, and we denote the Hodge dual $4$-form $\psi=\star\varphi\in \Omega^4$. 

To fix conventions we record certain algebraic identities for a $\rG_2$-structure \cite{Bryant,KarigiannisFlowsI}:
\begin{align}
\label{eq: varphi1 varphi1}
    \varphi_{abc}\varphi_{ijc}=g_{ai}g_{bj}-g_{aj}g_{bi}+\psi_{abij},\\
    \varphi_{abc}\varphi_{abc}=42, \quad 
    \varphi_{abc}\varphi_{ibc}=6g_{ai}. \label{eq: varphi2 varphi2}
\end{align}
\begin{align}
\psi_{abcd}\psi_{abmn}
    &=  4g_{cm}g_{dn}-4g_{cn}g_{dm}+2\psi_{cdmn},\label{eq: psi2 psi2}\\
    \psi_{abcd}\psi_{abcd}
    &=168, \quad 
    \psi_{abcd}\psi_{mbcd}
    = 24g_{am}\label{eq: psi3 psi3}.
\end{align}
\begin{align}
    \varphi_{ipq}\psi_{ijkl}
    &=  g_{pj}\varphi_{qkl}-g_{jq}\varphi_{pkl}+g_{pk}\varphi_{jql}  -g_{kq}\varphi_{jpl}+g_{pl}\varphi_{jkq}-g_{lq}\varphi_{jkp}.
    \label{eq: varphi1 psi1}
\end{align}

The action of $\rG_2$ on the space of differential forms $\Omega^k$ determines irreducible representations of $\rG_2$, e.g.
\begin{equation*}
    \Omega^4=\Omega^4_1\oplus \Omega_7^4\oplus \Omega^4_{27}, \qandq \Omega^5=\Omega^5_7\oplus \Omega^5_{14},
    \end{equation*}
where 
\begin{align*}
    \Omega^5_7=&\ \{\beta\in \Omega^5: \quad \star(\varphi\wedge\star\beta)=2\beta\}=\{X^\sharp \wedge \psi: \quad X\in\sX(M)\},\\
    \Omega^5_{14}=&\ \{\beta\in \Omega^5: \quad \star(\varphi\wedge\star\beta)=-\beta\}=\{\beta\in \Omega^5: \quad \psi\wedge\star\beta=0\},\\
    \Omega^4_1=&\ \{f\psi\in \Omega^4: \quad f\in C^\infty(M,\bR)\},\\
    \Omega^4_7=&\ \{X^\sharp\wedge\varphi\in\Omega^4: \quad X\in\sX(M)\},\\
    \Omega^4_{27}=&\ \{\gamma\in\Omega^4: \quad \varphi\wedge \star\gamma=0 \qandq \psi\wedge \star\gamma=0\}.
\end{align*}
The respective decompositions of $\Omega^2$ and $\Omega^3$ are given by the Hodge star isomorphism. 

The $\GL(\bR^7)$-infinitesimal action on $\Omega^k$ gives an alternative description of the space of three and four-forms: the maps $\cdot\diamond\varphi:\End(TM)\rightarrow\Omega^3$ and $ \cdot\diamond\psi:\End(TM)\rightarrow\Omega^4$ defined by
\begin{align}\label{eq: gl-action}
    A\diamond \varphi= \left.\frac{d}{dt} \right|_{t=0}(e^{tA})^*\varphi, \qandq A\diamond \psi= \left.\frac{d}{dt} \right|_{t=0}(e^{tA})^*\psi.
\end{align}
Using local coordinates $\{x_1,...,x_7\}$ on $M$, a $k$-form $\alpha$ can be written as
\begin{equation*}
    \alpha=\frac{1}{k!}\alpha_{i_1,...,i_k}dx^{i_1\cdots i_k}.
\end{equation*}
In coordinates the maps of \eqref{eq: gl-action} are given by
\begin{equation}\label{eq: diamond operator}
\begin{split}
(A\diamond\varphi)_{ijk}=&\ A_i^l\varphi_{ljk}+A_j^l\varphi_{ilk}+A_k^l\varphi_{ijl} \quad \text{where} \quad A_i^l=A_{im}g^{lm}\\
(A\diamond\psi)_{ijkl}=&\ A_i^m\psi_{mjkl}+A_j^m\psi_{imkl}+A_k^m\psi_{ijml}+A_l^m\psi_{ijkm}.
\end{split}
\end{equation}
Since the stabilizer of $\varphi$ and $\psi$ are $\rG_2$ and $\rG_2\times \mathbb{Z}_2$, respectively, then 
$$\ker(\cdot\diamond\varphi)=\ker(\cdot\diamond\psi)=\fg_2=\mathrm{Lie}(\rG_2)\simeq \Omega^2_{14}.
$$
The above together with the $\rG_2$-equivariance of \eqref{eq: gl-action} and Schur's lemma give the identifications 
\begin{equation*}
     \bR\cdot g\simeq\Omega^k_1, \quad \Omega^2_7\simeq\Omega^k_7, \quad S^2_0(TM)\simeq\Omega^k_{27}, \qforq k=3,4,
\end{equation*}
where $S^2_0(TM)$ denotes traceless symmetric $2$-tensors. The next lemma collects some properties of \eqref{eq: gl-action}:

\begin{lemma}\cite[Eq. (2.20), Cor. 2.30, 2.38]{DwivediGianniotisKarigiannis}
    Let $A\in \End(TM)$ and $\varphi$ a $\rG_2$-structure on $M$. Denote the $\rG_2$-decomposition of $A=A_1+A_{27}+A_7+A_{14}\in \bR\cdot g\oplus S^2_0(TM)\oplus \Omega_7^2\oplus\Omega^2_{14}$, then:
    \begin{enumerate}
        \item If $A_7=X\lrcorner \varphi$ for $X\in\sX(M)$ then 
        $$
        (X\lrcorner \varphi)\diamond\varphi=3X\lrcorner\psi \qandq (X\lrcorner\varphi)\diamond\psi=-3X\wedge\varphi.
        $$
        \item $A_{14}\diamond\varphi=0$ and $A_{14}\diamond\psi=0$.
        \item $A_1\diamond \varphi=\frac{3}{7}\tr(A)\varphi$ and $A_1\diamond\psi=\frac{4}{7}\tr(A)\psi$.
        \item $\star(A\diamond\varphi)=\left(\frac34 A_1-A_{27}+A_7\right)\diamond\psi$.
    \end{enumerate}
\end{lemma}

For later use, we consider the operator 
\begin{equation}\label{eq: map P}
  P: A=(A_{ij})\in \End(TM)\mapsto  A\lrcorner\psi=(A_{ab}\psi_{abij})\in \fso(TM).
\end{equation}
In local coordinates, for $\beta=\frac{1}{2}\beta_{ij}dx^{ij}$ and using the relation
$$
  \star(\alpha\wedge\gamma)=(-1)^k\alpha\lrcorner \star\gamma \qforq \alpha\in\Omega^1 \qandq \gamma\in \Omega^k,
$$
we obtain 
$$
\star(\beta\wedge\varphi)=\frac14\beta_{ij}\psi_{ijkl}dx^{kl}=\frac{1}{2}P\beta.
$$
In particular, 
\begin{align*}
    \beta\in\Omega^2_7 \quad \Leftrightarrow \quad P\beta=4\beta \qandq 
    \beta\in\Omega^2_{14} \quad \Leftrightarrow \quad P\beta=-2\beta.
\end{align*}
With respect to the induced metric on $\End(TM)$ (i.e.\ $\langle A,B\rangle=A_{ij}B_{ij}$), by \eqref{eq: psi2 psi2} we have
 \begin{equation}\label{eq: P norm}
     \langle PA,PB\rangle=4\langle A,B\rangle-4\langle A,B^t\rangle+2\langle A,PB\rangle.
\end{equation}

\subsection{Torsion of \texorpdfstring{$\rG_2$}{G2}-structures}

The first order differential invariant of a $\rG_2$-structure is given by $\nabla\varphi$, which is equivalently described by the so-called \emph{full torsion tensor} $T \in \End(TM)$ of $\varphi$, defined by
 \begin{equation}\label{eq: T=nabla_varphi}
 T_{ab}=\frac{1}{24}\nabla_a\varphi_{ijk}\psi_{bijk}, \quad (\text{equivalently}, \quad \nabla_a\varphi_{ijk}=T_{ab}\psi_{bijk}),
 \end{equation}
 and similarly
 \begin{align}\label{eq: nabla psi}
     \nabla_i\psi_{jklm}=-T_{ij}\varphi_{klm}+T_{ik}\varphi_{jlm}-T_{il}\varphi_{jkm}+T_{im}\varphi_{jkl}.
 \end{align}
 According to the $\rG_2$-irreducible decomposition $\End(TM)$, the tensor $T$ admits a decomposition into torsion components
\begin{equation}\label{eq: torsion components}
    T=T_1+T_{27}+T_7+T_{14},
\end{equation}
where $T_1+T_{27}\in \bR\cdot g\oplus S^2_0(TM)=S^2(TM)$ and $T_7+T_{14}\in  \bR^7\oplus \fg_2=\fs\fo(7)$ are the symmetric and skew-symmetric part of $T$, respectively. 

Similarly, the irreducible components of $d\varphi\in \Omega^4$ and $d\psi\in \Omega^5$ are determined by the (uniquely determined) \emph{torsion forms} $\tau_k\in \Omega^k$ (for $k=0,1,2,3$) as
\begin{align}
    \label{eq:d:varphi}
        d\varphi&=
        \tau_0\psi+3\tau_1\wedge \varphi+\star\tau_3,
        \\\label{eq:d:varpsi}
        d\psi&= 4\tau_1\wedge \psi+\tau_2\wedge \varphi.
\end{align}
The torsion tensor $T$ can be written explicitly in terms of the torsion forms $\tau_k$.
\begin{equation}\label{eq: full torsion tensor}
    T=\frac{\tau_0}{4}g-\tau_{27}-\tau_1\lrcorner\varphi-\frac{1}{2}\tau_2,
\end{equation} 
where $\tau_3=\tau_{27}\diamond\varphi$.

Using now the operator $P$ in \eqref{eq: map P}, we have
$PT=-4\tau_1\lrcorner\varphi+\tau_2$ and also that:
 \begin{gather}
 \begin{split}\label{eq:trT}
     \tau_0=&\ \frac{4}{7}\tr T,\\
     \tau_{27}=&\ -\frac{1}{2}(T+T^t)+\frac{1}{7}(\tr T)g,\\
     \tau_1\lrcorner\varphi=&\ -\frac{1}{6}(T-T^t)-\frac{1}{6}PT,\\ 
     \tau_2=&\ -\frac{2}{3}(T-T^t)+\frac{1}{3}PT.
 \end{split}
 \end{gather}
 With respect to the induced metric on $\End(TM)$, by \eqref{eq: psi2 psi2} and \eqref{eq: P norm} we have
\begin{align} \nonumber
    \tau_0^2=&\ \frac{16}{49}(\tr T)^2\\ \nonumber
     |\tau_{27}|^2=&\ \frac{1}{2}|T|^2+\frac12\langle T,T^t\rangle-\frac{1}{7}(\tr T)^2\\ \label{eq: norm torsion forms}
     |\tau_1\lrcorner\varphi|^2=&\ \frac{1}{6}|T|^2-\frac{1}{6}\langle T,T^t\rangle+\frac{1}{6}\langle T,PT\rangle\\ \nonumber
     |\tau_2|^2=&\ \frac{4}{3}|T|^2-\frac{4}{3}\langle T,T^t\rangle-\frac{2}{3}\langle T,PT\rangle. \nonumber
\end{align}
Here, the left-hand-side denotes the norm on $\End(TM)$ (e.g.\ $|\tau_1\lrcorner\varphi|^2=(\tau_1\lrcorner\varphi)_{ij}(\tau_1\lrcorner\varphi)_{ij}$). Finally, using \eqref{eq: varphi1 varphi1} and \eqref{eq: varphi2 varphi2}, we obtain
\begin{equation*}
    |\tau_1\lrcorner \varphi|^2=6|\tau_1|^2 \qandq |\tau_{27}|^2=\frac12|\tau_3|^2,
\end{equation*}
where the right-hand-side denotes the induced norm on $\Omega^k$, i.e.\
\begin{equation}\label{eq:normkform}
|\gamma|^2=\frac{1}{k!}\gamma_{i_1...i_k}\gamma_{i_1...i_k} \qforq \gamma\in \Omega^k.
\end{equation}

The second order differential invariants of a $\rG_2$-structure are given by $\nabla\nabla\varphi\in\Omega^1\otimes\Omega^1(M,\Omega^2_7)$, which is equivalent to $\nabla T\in\Omega^1(M,\End(TM))$. Similar to the decomposition of $T$ into $\rG_2$-irreducible components \eqref{eq: torsion components}, the decomposition of $\nabla T$ can be studied using representation theory \cite[$\S$ 4.2]{DwivediGianniotisKarigiannis}. However, using the Ricci identity for $\nabla_i\nabla_j\varphi$, together with \eqref{eq: T=nabla_varphi} and \eqref{eq: nabla psi} gives the so-called \emph{$\rG_2$-Bianchi identity}~\cite[Theorem 4.2]{KarigiannisFlowsI}:
\begin{equation}\label{eq: G2Bianchi_ident}
    \nabla_i T_{jk}-\nabla_j T_{ik}=\left(\frac12 R_{ijmn}- T_{im}T_{jn}\right)\varphi_{mnk}.
\end{equation}
Define the tensors $$
\curl T_{ab}=\nabla_pT_{aq}\varphi_{pqb}, \quad (T\lrcorner \varphi)_a=T_{bc}\varphi_{abc} \qandq (T_\sym)_{ab}=\frac12\left(T_{ab}+T_{ba}\right).
$$

As a first consequence of \eqref{eq: G2Bianchi_ident}, we obtain expressions for the Ricci tensor and the scalar curvature of $\varphi$:

\begin{lemma}\cite[Prop 4.15]{KarigiannisFlowsI}\label{lm: Ricci and scalar curv}
    Given a $\rG_2$-structure $\varphi$ on $M$, the Ricci and the scalar curvature are given by\begin{align*}
        \Ric_{ij}=&\ \left(\nabla_iT_{mn}-\nabla_mT_{in}\right)\varphi_{jmn}+\tr(T)T_{ij}-T_{im}T_{mj}+T_{im}T_{np}\psi_{mnpj},\\
        R=&\ 2\nabla_iT_{mn}\varphi_{imn}+(\tr T)^2-T_{im}T_{mi}+T_{im}T_{np}\psi_{mnpi}.
    \end{align*}
    Equivalently,
    \begin{align*} 
        \Ric=&-(\curl T)_\sym+\frac12\cL_{T\lrcorner\varphi}g+(\tr T)T_\sym-T^2_\sym,\\ 
        R=&-2d^{\star}(T\lrcorner\varphi)+\langle T,PT\rangle+(\tr T)^2-\langle T,T^t\rangle.
    \end{align*}
    Further using \eqref{eq: norm torsion forms}, the previous formula for $R$ becomes (cf.\ \cite[(4.28)]{Bryant2006})
    \begin{equation*}
        R=12d^{\star}\tau_1+\frac{21}{8}\tau_0^2+30|\tau_1|^2-\frac12|\tau_2|^2-\frac12|\tau_3|^2.
    \end{equation*}
\end{lemma}

\subsection{Integrable \texorpdfstring{$\rG_2$}{G2}-structures}\label{sec: Ivanov-Stanchev}

A $\rG_2$-structure is called \emph{integrable} when $\tau_2=0$. They play a distinguished role, as they admit compatible connections with totally skew-symmetric torsion.  
\begin{prop}[\cite{friedrich2003parallel_string_theory}]
\label{prop:FIG2}
    Let $(M^7,\varphi)$ be a 7-manifold endowed with a $\rG_2$-structure $\varphi$. Then, there exists an affine connection $\nabla^+$ with totally skew-symmetric torsion preserving $\varphi$, that is, $\nabla^+ \varphi = 0$, if and only if $\tau_2=0$. In this case, we say $\varphi$ is integrable, and the connection $\nabla^+$ is unique and given by
\begin{equation}
\label{eq: unique conn and tor}
    \nabla^+ = \nabla + \frac{1}{2}g^{-1}H_\varphi, \qquad 
    H_\varphi = \frac{1}{6}\tau_0\varphi - \tau_1\lrcorner\psi - \tau_3,
\end{equation}
where $\nabla$ is the Levi-Civita connection of $g = g_\varphi$ and $H_\varphi \in \Omega^3$ is the torsion of $\nabla$. The connection $\nabla^+$ is called the \emph{characteristic connection} of $\varphi$. 
\end{prop}

In the next result we recall a remarkable formula due to Ivanov-Stanchev, which relates the four-form $dH_\varphi$ with the components of the Ricci tensor of the characteristic connection $\nabla^+$.

\begin{prop}[\cite{ivanov2025}]
\label{prop:ISG2}
For $\varphi$ integrable, the Ricci tensor $\Ric^{+} = \Ric^{\nabla^+}$ of the characteristic connection $\nabla^+=\nabla+\frac{1}{2}g^{-1}H_\varphi$ satisfies 
\begin{equation}\label{eq: Ivanov formula}
\Ric_{ij}^+ + 4\nabla^+_i(\tau_1)_j=\frac{1}{12}(dH_\varphi)_{iabc}\psi_{jabc}.
\end{equation}
Furthermore, the symmetric and skew-symmetric part of \eqref{eq: Ivanov formula}
are
\begin{align}\label{eq: Ivanov formula_sym}
    \Ric_{ij}-\frac{1}{4}(H_\varphi^2)_{ij}+2\cL_{\tau_1}g_{ij}=\frac{1}{24}\left((dH_\varphi)_{iabc}\psi_{jabc}+(dH_\varphi)_{jabc}\psi_{iabc}\right),\\ \label{eq: Ivanov formula_skewsym}
    -\frac12 (d^{\star}H_\varphi)_{ij}+2(d\tau_1)_{ij}-2(\tau_1\lrcorner H_\varphi)_{ij}=\frac{1}{24}\left((dH_\varphi)_{iabc}\psi_{jabc}-(dH_\varphi)_{jabc}\psi_{iabc}\right).
\end{align}
\end{prop}

Using the previous result, we obtain the following decomposition of the four-form $dH_\varphi$ into $7$ and $1 + 27$ components.

\begin{prop}\label{prop:dH}
    For $\varphi$ integrable, the $4$-form $dH_\varphi$ has the decomposition
    \begin{align}\label{eq: components dH}
        dH_\varphi=&-\frac13 d(\tr T)\wedge\varphi+\left(\Ric-\frac{1}{4}(H_\varphi^2)+2\cL_{\tau_1}g-\frac{1}{8}(R_g-4d^{\star}\tau_1-\frac{3}{2}|H_\varphi|^2)g\right)\diamond\psi.
    \end{align}
    In particular,
    \begin{align}\label{eq: pi1_dH}
        \langle dH_\varphi,\psi\rangle = \ 4d^{\star}\tau_1+\frac{49}{36}\tau_0^2+16|\tau_1|^2-|H_\varphi|^2.
    \end{align}
\end{prop}
\begin{proof}
Consider each irreducible component of the $4$-form $dH_\varphi=a\psi+X\wedge\varphi+*(S\diamond\varphi)$, where
\begin{equation*}
    a=\frac17\langle dH_\varphi,\psi\rangle, \quad X^i=\frac14\langle dH_\varphi,e^i\wedge \varphi\rangle, \quad S_{ij}=2ag_{ij}-\frac{1}{24}\left((dH_\varphi)_{iabc}\psi_{jabc}+(dH_\varphi)_{jabc}\psi_{iabc}\right)
\end{equation*}
Taking the trace of \eqref{eq: Ivanov formula_sym}, we get \begin{align}\label{eq: auxiliar pi1dH}
    R_g-4d^{\star}\tau_1-\frac{3}{2}|H_\varphi|^2 =&  \frac{1}{12}(dH_\varphi)_{iabc}\psi_{iabc}
     =  2\langle dH_\varphi,\psi\rangle=  14a.
\end{align}
Now, using $*(S\diamond\varphi)=-(S\diamond\psi)$ we obtain $a\psi+*(S\diamond\varphi)=(\frac{a}{4}g-S)\diamond\psi$
$$
  \frac{a}{4}g-S=-\frac74 a g+\Ric-\frac{1}{4}(H_\varphi^2)+2\cL_{\tau_1}g.
$$
On the other hand, using \eqref{eq: varphi1 psi1} we have
\begin{align*}
    \frac{1}{24}\left((dH_\varphi)_{iabc}\psi_{jabc}-(dH_\varphi)_{jabc}\psi_{iabc}\right)\varphi_{ijk}=&-\frac14(dH_\varphi)_{kibc}\varphi_{ibc}\\
    =&-\frac32\langle dH_\varphi,e_k\wedge\varphi\rangle=-6 X_k.
\end{align*}
Thus, by \eqref{eq: Ivanov formula_skewsym}
$$
  X_k=\frac16\left(\frac12 (d^{\star}H_\varphi)_{ij}-2(d\tau_1)_{ij}+2(\tau_1\lrcorner H_\varphi)_{ij}\right)\varphi_{ijk}.
$$
Since
\begin{align*}
\frac12d^{\star}H_\varphi-2d\tau_1+2\tau_1\lrcorner H_\varphi=&\ -\frac{7}{12} \star(d\tau_0\wedge\psi)-2d\tau_1-2\star(d\tau_1\wedge\varphi)\\
=&\ -\frac{1}{3}(d\tr T)\lrcorner\varphi-2d\tau_1-2d\tau_1\lrcorner\psi,
\end{align*}
then, by \eqref{eq: varphi1 varphi1}
\begin{equation}\label{eq: 7 piece of dH}
  X_k=-\frac13 d(\tr T)_k-\frac53 d\tau_1\lrcorner\varphi.
\end{equation}
Notice that $d\tau_1\lrcorner\varphi=\star(d\tau_1\wedge\psi)=0$, since $d\psi=4\tau_1\wedge\psi$. Finally, replacing in \eqref{eq: auxiliar pi1dH} the expression for the scalar curvature given Lemma \ref{lm: Ricci and scalar curv}, we obtain \eqref{eq: pi1_dH}.
\end{proof}

\subsection{Variation of \texorpdfstring{$\rG_2$}{G2}-structures}

Any deformation of a $\rG_2$-structure $\varphi$ on $M$ has the form \cite{KarigiannisFlowsI}
\begin{equation}\label{eq: general_G2_flow}
    \frac{\partial}{\partial t}\varphi=S\diamond\varphi+X\lrcorner\psi,
\end{equation}
where $S\diamond \varphi$ is the $3$-form \eqref{eq: diamond operator}, $S\in S^2(TM)$ is a symmetric $2$-tensor and $X\in \sX(M)$ is a vector field on $M$. In particular, the evolution of $g$, $\vol$, $\nabla$ and $\psi$ are  
\begin{align}\label{eq: ddt g}
   \frac{\partial}{\partial t}g_{ij}=&\ 2S_{ij}, \quad \frac{\partial}{\partial t}g^{ij}=-2S^{ij}, \quad \frac{\partial}{\partial t}\vol=\tr(S)\vol,\\ \label{eq: ddt Gamma}
   \frac{\partial}{\partial t}\Gamma_{ij}^k=&\ g^{kl}(\nabla_iS_{jl}+\nabla_jS_{il}-\nabla_lS_{ij}),\\ \label{eq: ddt psi}
   \frac{\partial}{\partial t}\psi=&\ S\diamond\psi-X\wedge\varphi,
\end{align}
where $\Gamma_{ij}^k$ denotes the Christoffel symbols of $\nabla$. 

\begin{lemma}
    If $g$ and $\vol$ evolve by \eqref{eq: ddt g}, then for any time-dependent $\alpha\in \Omega^k$ we have
    \begin{align}\label{eq: ddt *}
        \star\frac{\partial }{\partial t}\alpha=\frac{\partial }{\partial t}\star\alpha+2\star(S\diamond\alpha)-\tr(S)\star\alpha.
    \end{align}
\end{lemma}

\begin{proof}
    Taking the evolution of $\alpha\wedge\star\alpha=g(\alpha,\alpha)\vol$, on the left-hand side we have
    \begin{align*}
        \partial_t(\alpha\wedge \star\alpha)=&\partial_t\alpha\wedge\star\alpha+\alpha\wedge\partial_t\star\alpha\\
        =&\alpha\wedge(\star\partial_t\alpha+\partial_t\star\alpha).
    \end{align*}
    On the other hand
    \begin{align*}
        \partial_t(g(\alpha,\alpha)\vol)=&\left(2g(\partial_t\alpha,\alpha)-2g(S\diamond \alpha,\alpha)+\tr(S)g(\alpha,\alpha)\right)\vol\\
        =&\alpha\wedge\left(2\star\partial_t\alpha-2\star(S\diamond \alpha)+\tr(S)\star\alpha\right).
    \end{align*}
    Thus, for any $\alpha\in\Omega^k$ 
    $$
      \alpha\wedge\left(-\star\partial_t\alpha+\partial_t\star\alpha+2\star(S\diamond \alpha)-\tr(S)\star\alpha\right)=0,
    $$
    hence $\star\partial_t\alpha=\partial_t\star\alpha+2\star(S\diamond \alpha)-\tr(S)\star\alpha$.
\end{proof}

Notice that \eqref{eq: ddt *} provides an alternative proof of \eqref{eq: ddt psi} using \eqref{eq: general_G2_flow}.

\begin{lemma}\cite[Lemma 3.6, Theorem 3.7]{KarigiannisFlowsI}\label{lemma evolution T general}
    The evolution of $\nabla\varphi$ and $T$ under the flow \eqref{eq: general_G2_flow} are
    \begin{align}\label{eq: ddt_nabla_varphi}
        \partial_t\nabla_i\varphi_{jkl}=&\ (S\diamond\nabla_i\varphi)_{jkl}+(X\lrcorner\nabla_i\psi)_{jkl}+(\nabla_iX\lrcorner\psi)_{jkl}-(\nabla_jS_{mi}-\nabla_mS_{ji})\varphi_{mkl},\\ \nonumber
        &\ -(\nabla_kS_{mi}-\nabla_mS_{ki})\varphi_{jml}-(\nabla_lS_{mi}-\nabla_mS_{li})\varphi_{jkm}\\\label{eq: ddt T}
        \partial_tT_{ij}=&\ T_{ik}S_{jk}+T_{ik}(X\lrcorner\varphi)_{jk}+\nabla_iX_j-\nabla_kS_{il}\varphi_{jkl}.
    \end{align}
\end{lemma}

\section{The heterotic \texorpdfstring{$\rG_2$}{G2} flow}
\subsection{Generic heterotic flow equations}

Let $M$ be an oriented, spin manifold of dimension seven. In this section we introduce a coupled flow for one-parameter families of $\rG_2$-structures $\varphi_t \in \Omega^3$ and smooth functions $\phi_t \in C^\infty(M,\bR)$ on $M$. In the sequel, due to the relation to the geometry of string theory, we will refer to the function $\phi$ as the \emph{dilaton}.

\begin{defn} \label{d:flowdef}
Given an oriented, spin $7$-manifold $M$, a one-parameter family of $\rG_2$-structures $\varphi_t \in \Omega^3$ and dilaton functions $\phi_t \in C^\infty(M,\bR)$ satisfies the \emph{generic heterotic $\rG_2$ flow} if
\begin{gather}\label{eq: mod_anomaly_flowCbeta}
\begin{split}
    \frac{\partial}{\partial t}\left(e^{-4\phi}\psi\right)=&-dH_{\varphi}+\frac{7}{4}C d\left(\tau_0\varphi \right),\\
    \frac{\partial}{\partial t}\phi=&\ e^{4\phi} \left(\Delta_g \phi - \gamma |d\phi|^2 + \frac{1}{4}|\tau_3|^2 - \sigma \tau_0^2 \right),
\end{split}
\end{gather}
where $H_\varphi$ is as in \eqref{eq: unique conn and tor}, and $C, \gamma,\sigma \in \bR$ are real parameters.
\end{defn}

\noindent Obviously, the flow \eqref{eq:heteroticG2flowintro} corresponds to the generic heterotic $\rG_2$ flow \eqref{eq: mod_anomaly_flowCbeta} with $C = -4/3$. The second term in the evolution equation for $e^{-4\phi}\psi$, depending on the parameter $C \in \mathbb{R}$, is inspired by Grigorian's modification of the Laplacian co-flow for $\rG_2$-structures.

A key elementary feature of this flow is that it preserves the condition of being conformally coclosed.  That is, 
\begin{equation}\label{eq:confcoclosed}
d(e^{-4\phi}\psi) = 0
\end{equation}
is preserved along the flow.  Notice that in particular a flow line with conformally coclosed initial data will remain  integrable along the flow, i.e.\ $\tau_2 = 0$ will hold.  Furthermore  such flow lines will satisfy
$$
\tau_1 = d\phi_t.
$$
We next compute the induced evolution equations for the induced metric, and fundamental three-form and four-form:


\begin{prop}\label{p:evolutiongvarphiCbeta}
Let $(\varphi_t,\phi_t)$ be a solution to \eqref{eq: mod_anomaly_flowCbeta}, with conformally coclosed initial condition $(\varphi_0,\phi_0)$, that is, $d(e^{-4\phi_0}\psi_0) = 0$. Then, the evolution of $\psi_t$, $\varphi_t$, and $g_t$ are:
\begin{align*}
\frac{\partial}{\partial t} g =&\ e^{4\phi}\left(-2\Ric_{g}+\frac12 H_{\varphi}^2-4\cL_{d\phi}g + \frac{7}{2}C\tau_0 T_\sym+\left(\left(\frac{7}{12}-2\sigma\right)\tau_0^2+2\left(3-\gamma\right)|d\phi|^2 \right)g\right),\\
\frac{\partial}{\partial t}\varphi  
=&\ - \frac{7}{12}e^{4\phi}\left((1 + 3C)d\tau_0 + 9 C\tau_0 d\phi\right)\lrcorner\psi\\
    &\ +e^{4\phi}\left(-\Ric_{g}+\frac14H_{\varphi}^2-2\cL_{d\phi}g + \frac{7}{4}C\tau_0 T_\sym +\left(\left(\frac{7}{24}-\sigma\right)\tau_0^2+\left(3-\gamma\right)|d\phi|^2 \right)g\right)\diamond \varphi\\
\frac{\partial}{\partial t}\psi   =&\ \frac{7}{12}e^{4\phi}\left((1 + 3C)d\tau_0 + 9 C\tau_0 d\phi\right)\wedge\varphi\\
    &\ +e^{4\phi}\left(-\Ric_{g}+\frac14H_{\varphi}^2-2\cL_{d\phi}g + \frac{7}{4}C\tau_0 T_\sym +\left(\left(\frac{7}{24}-\sigma\right)\tau_0^2+\left(3-\gamma\right)|d\phi|^2 \right)g\right)\diamond \psi,
\end{align*}
where $T_\sym=\frac{\tau_0}{4} g -\tau_{27}$. 
\end{prop}

\begin{proof}
Since $e^{-4\phi}\psi$ evolves by an exact $4$-form and its initial condition is closed, $\varphi$ is an integrable $\rG_2$-structure with $d\varphi=\tau_0\psi+3d\phi\wedge\varphi+\star\tau_3$. Thus, from Proposition \ref{prop:dH},
    \begin{align*}
        -dH_\varphi+\frac{7}{4}Cd\left(\tau_0\varphi\right)=&\ \frac{7}{12} d\tau_0\wedge\varphi+\left(-\Ric+\frac{1}{4}(H_\varphi^2)-2\cL_{d\phi}g+\frac{1}{4}\langle dH_\varphi,\psi \rangle g\right)\diamond\psi\\
        &+\frac74C\left( d\tau_0\wedge\varphi+\tau_0\left(3d\phi\wedge\varphi+\left(\frac{\tau_0}{4}g-\tau_{27}\right)\diamond\psi\right)\right),
    \end{align*}
    using \eqref{eq: pi1_dH} with $\tau_1=d\phi$, we get
    \begin{align}\label{eq:dH1explicit}
        \langle dH_\varphi,\psi\rangle=& -4\Delta_g\phi+\frac76\tau_0^2+12|d\phi|^2-|\tau_3|^2.
    \end{align}
    Hence, the evolution for $\psi$ is
    \begin{align*}
        \frac{\partial}{\partial t}\psi =&-e^{4\phi}dH_\varphi+\frac74 C e^{4\phi}d(\tau_0\varphi)+4\left(\frac{\partial}{\partial t}\phi\right) \psi\\
        =&\ \frac74e^{4\phi}\left(\left(\frac13+C\right)d\tau_0 +3\tau_0 d\phi\right)\wedge\varphi\\
    &+e^{4\phi}\left(-\Ric_{g}+\frac14H_{\varphi}^2-2\cL_{d\phi}g-\frac74 C\tau_0 \tau_{27}+\left(\left(\frac{7}{24}+\frac{7}{16}C-\sigma\right)\tau_0^2+\left(3-\gamma\right)|d\phi|^2 \right)g \right)\diamond \psi.
    \end{align*}
The result follows from \eqref{eq: general_G2_flow} and \eqref{eq: ddt g}, taking $T_\sym=\frac{\tau_0}{4} g -\tau_{27}$.
\end{proof}

\subsection{Relation to the large volume heterotic $\rG_2$ system}

In this section we characterize fixed points for the heterotic $\rG_2$ flow \eqref{eq: mod_anomaly_flowCbeta} with conformally coclosed initial condition. We will make specific choices of the parameters $\gamma$ and $\sigma$, while keeping $C \neq 0$ generic, in order to make contact with the \emph{large volume heterotic $\rG_2$ system}.

\begin{defn}
\label{def:modheteroticG2}
We say that a pair $(\varphi,\phi)$, given by a $\rG_2$-structure $\varphi$ on $M$ and a smooth function $\phi \in C^\infty(M,\bR)$ (the \emph{dilaton}), is a solution of the \emph{modified heterotic $\rG_2$-system} with constant $C \in \bR$ if the following equations are satisfied 
\begin{equation}
\label{eq:modheteroticG2}
\begin{split}
    d (e^{-4\phi}\psi) & = 0,\\
    d\left(H - \frac{7}{4}C\left(\tau_0\varphi \right)\right)
    & = 0,\\
    \langle dH,\psi\rangle & = 0,
\end{split}
\end{equation}
where
\begin{equation}
\label{eq:Killinghetcompatiblephi}
H = H_\varphi = \frac{1}{6}\tau_0\varphi - \tau_1\lrcorner\psi - \tau_3.
\end{equation}
\end{defn}

\noindent Note that conformally by \eqref{eq: pi1_dH} coclosed fixed points of the flow \eqref{eq: mod_anomaly_flowCbeta} correspond to solutions of \eqref{eq:modheteroticG2} with
\begin{equation}\label{eq:sigmagammahet}
\gamma = 3, \qquad \sigma 
= \frac{42}{144}.
\end{equation}
Our goal is to prove that, for $C \neq 0$, solutions of \eqref{eq:modheteroticG2} on a compact manifold are torsion-free $\rG_2$-structures. For the proof we will need the following technical lemma:

\begin{lemma}\label{lem:Lich1G2}
Let $\varphi$ be an integrable $\rG_2$-structure. Taking $H = H_\varphi$, one has
\begin{gather}\label{eq:magicformula}
\begin{split}
\frac{49}{36}\tau_0^2 = \mathcal{S}^+_{\varphi} - \langle dH,\psi\rangle,
\end{split}
\end{gather}
where $\mathcal{S}^+_{\varphi}$ is the \emph{generalized scalar curvature} of $\varphi$, defined by
\begin{equation}\label{eq:Genscalar}
\mathcal{S}^+_{\varphi} = R_{g}-\frac{1}{2}|H|^2 - 8 d^{\star} \tau_1 - 16\left|\tau_1\right|^2 \in C^\infty(M,\bR).
\end{equation}
\end{lemma}

\begin{proof}

From Lemma \ref{lm: Ricci and scalar curv}, the scalar curvature can be written as
$$
  R_g= 12d^{\star}\tau_1+\frac{49}{18}\tau_0^2+32|\tau_1|^2-\frac12|H|^2.
$$
Hence, the relation \eqref{eq:magicformula} follows from \eqref{eq: pi1_dH} (cf. \eqref{eq:dH1explicit}) and \eqref{eq:Genscalar}.    
\end{proof}

\begin{rmk}
The generalized scalar curvature plays an important role in the theory of generalized Ricci flow, being closely related to the volume density of the generalized Perelman energy functional \cite{GRFbook}.  In \eqref{eq:Genscalar} we use the Hodge norm on differential forms, given by \eqref{eq:normkform}.
\end{rmk}

\begin{rmk}
There exists a spinorial proof of the previous result, by adaptation of the methods in
\cite[Proposition 3.6]{SGFLS}.
\end{rmk}

With these preliminaries in place we can prove the main result of this section, which characterizes solutions of the modified heterotic $\rG_2$-system \ref{eq:modheteroticG2}.  In particular we show that the torsion is closed, hence by the recent result of Ivanov-Stanchev \cite{ivanov2025} defines a generalized Ricci soliton.  Due to further constraints on the generalized scalar curvature, if the manifold is compact then in fact it defines a torsion-free $\rG_2$-structure.

\begin{prop}\label{prop:weakG2}
Let $(\varphi,\phi)$ be a solution of the equations \eqref{eq:modheteroticG2} with $0 \neq C \in \mathbb{R}$. Then,
\begin{equation}\label{eq:vanishing}
\mathcal{S}^+_{\varphi} = 0, \qquad \tau_0 = 0, \qquad d H = 0,
\end{equation}
where $\mathcal{S}^+_{\varphi}$ is the generalized scalar curvature
$$
\mathcal{S}^+_{\varphi} = R_g-\frac{1}{2}|H_\varphi|^2 - 8 d^{\star}d\phi  - 16\left|d\phi\right|^2.
$$
Consequently, assuming that $M$ is compact, one has that $\varphi$ is a torsion-free $\rG_2$-structure, $d \phi = 0$ and $H_\varphi = 0$. 
\end{prop}

\begin{proof}
By the first and third equations in \eqref{eq:modheteroticG2} and Lemma \ref{lem:Lich1G2}, we have that
$$
\frac{49}{36}\tau_0^2 = \mathcal{S}^+_{\varphi}.
$$
Taking the $1$-dimensional component of the second equation in \eqref{eq:modheteroticG2}, we also obtain that
$$
-\frac{49}{4}C\tau_0^2 = -\frac{7}{4} C\tau_0 \langle d\varphi, \psi\rangle = - \langle dH, \psi\rangle = 0,
$$
and consequently \eqref{eq:vanishing} holds (since $C \neq 0$). From this, $dH = 0$ and by Proposition \ref{prop:ISG2} it follows that $(g_\varphi)$ is a gradient generalized Ricci soliton with vanishing generalized scalar curvature, and the statement follows assuming $M$ compact (see \cite[Proposition 5.8]{GF19} or \cite[Proposition 3.1]{Clarke2022} for a specific proof in the seven dimensional case).
\end{proof}

\subsection{Monotonicity of the \texorpdfstring{$\rG_2$}{G2}-dilaton functional}

In this section we give an alternative characterization of the fixed points for the generic heterotic $\rG_2$ flow, with conformally coclosed initial condition, for different values of the parameters $\gamma$ and $\sigma$ as those taken in Proposition \ref{prop:weakG2}. For this, we study monotonicity of a natural functional inspired by the variational approach to the six-dimensional heterotic system proposed in \cite{garciafern2018canonical}, better known as the Hull-Strominger system. Our analysis will illustrate the utility of relaxing our flow equations \eqref{eq:heteroticG2flowintro} in particular the evolution equation for the dilaton $\phi$. 
Using this, we will then prove a rigidity result, which shows that solutions of the system on a compact manifold are generically torsion-free $\rG_2$-structures.

\begin{defn} \label{d:dilatonfunctional}
Given an oriented, compact spin $7$-manifold $M$ and a pair $(\varphi,\phi)$, given by a $\rG_2$-structure $\varphi$ and smooth function $\phi \in C^\infty(M,\bR)$ on $M$, respectively, the \emph{$\rG_2$-dilaton functional} is defined by
\begin{gather}\label{eq:dilatonf}
\mathcal{M}(\varphi,\phi) = \int_M e^{-4\phi} \operatorname{Vol}_\varphi
\end{gather}
where $\operatorname{Vol}_\varphi$ is the volume form determined by the $\rG_2$-structure.
\end{defn}

\begin{rmk}
Up to scaling, the functional \eqref{eq:dilatonf} coincides with the \emph{volume scale} of a three-dimensional heterotic string compactification, previously considered in the string theory literature \cite{deIaOssa:2019cci}.
\end{rmk}

We fix a degree-four cohomology class, which we shall refer to as the \emph{balanced class}
$$
\tau \in H^4(M,\mathbb{R}).
$$
We are interested in studying monotonicity of the $\rG_2$-dilaton functional along the set of conformally coclosed pairs $(\varphi,\phi)$, with fixed balanced class $\tau$, that is,
$$
\mathcal{C}_\tau = \Big{\{}(\varphi,\phi) \; | \; d\left(e^{-4\phi}\psi\right) = 0, \; [e^{-4\phi}\psi] = \tau \Big{\}} \subset \Omega^3_+ \times C^\infty(M,\bR),
$$
where $\Omega^3_+ \subset \Omega^3$ denotes the open subspace of $\rG_2$-structures on $M$.

\begin{lemma}\label{lem:variation}
The variation of the $\rG_2$-dilaton functional along $\mathcal{C}_\tau$ is given by
$$
\delta \mathcal{M}_{|(\varphi,\phi)}(\dot \varphi,\dot \phi) = 3 \int \dot \phi e^{-4\phi} \operatorname{Vol}_\varphi + \frac{1}{4}\int_M \langle h,\star_\varphi d \varphi\rangle\operatorname{Vol}_\varphi,
$$
where
$$
\delta(e^{-4\phi}\psi) = dh, \quad h \in \Omega^3
$$
is the variation of $e^{-4\phi} \psi$.
\end{lemma}
\begin{proof}
The condition $(\varphi,\phi) \in \mathcal{C}_\tau$ implies $\delta(e^{-4\phi}\psi) = dh$, for some $h \in \Omega^3$. More explicitly,
$$
\dot \psi -4\dot \phi \psi = e^{4\phi}dh.
$$
We note that
$$
\operatorname{Vol}_\varphi = \frac{1}{7}\varphi \wedge \psi.
$$
Then, taking variations in \eqref{eq:dilatonf} we have
\begin{align*}
\delta \mathcal{M}_{|(\varphi,\phi)}(\dot \varphi,\dot \phi) & = -4 \int \dot \phi e^{-4\phi} \operatorname{Vol}_\varphi + \frac{1}{7}\int_M e^{-4\phi} (\dot \varphi \wedge \psi + \varphi \wedge \dot \psi)\\
& = \frac{1}{7}\int_M e^{-4\phi} \dot \varphi \wedge \psi + \frac{1}{7}\int_M \varphi \wedge dh.
\end{align*}
Using the operator
$$
J_7 \colon \Omega^3 \to \Omega^3,
$$
defined by
$$
J_7 \beta = \frac{4}{3}\beta_1 + \beta_7 - \beta_{27},
$$
we can write (see \cite[Lemma 20]{hitchin2000geometry})
$$
\dot \varphi = J_7^{-1}\star_\varphi \dot \psi = J_7^{-1}\star_\varphi \left(e^{4\phi}dh + 4\dot \phi \psi\right).
$$
From this, we conclude
\begin{align*}
\delta \mathcal{M}_{|(\varphi,\phi)}(\dot \varphi,\dot \phi) & = \frac{3}{28}\int_M e^{-4\phi} \left(e^{4\phi}dh + 4\dot \phi \psi\right) \wedge \varphi + \frac{1}{7}\int_M \varphi \wedge dh\\
& = 3\int_M \dot \phi e^{-4\phi}\operatorname{Vol}_\varphi + \frac{1}{4}\int_M  h \wedge d\varphi.
\end{align*}

\end{proof}

We need the following technical lemma, relating $\star_\varphi d \varphi$ with the characteristic torsion (cf.  \eqref{eq: unique conn and tor})
along $\mathcal{C}_\tau$.

\begin{lemma}\label{lem:H*dvarphi}
For $(\varphi,\phi) \in \mathcal{C}_\tau$, we have
$$
\star_\varphi d \varphi = -H_\varphi + \frac{7}{6}\tau_0 \varphi - 4 \nabla \phi \lrcorner \psi.
$$
\end{lemma}
\begin{proof}
Using that $\tau_1 = d\phi$ along $\mathcal{C}_\tau$, we obtain (see \eqref{eq:d:varphi})
\begin{align*}
\star_\varphi d \varphi = \tau_0\varphi+3\star_\varphi d\phi\wedge \varphi + \tau_3 = -H_\varphi + \frac{7}{6}\tau_0 \varphi - 4 \nabla \phi \lrcorner \psi.
\end{align*}
\end{proof}

Our next result proves the desired monotonicity for the $\rG_2$-dilaton functional \eqref{eq:dilatonf} along the generic heterotic $\rG_2$ flow \ref{eq: mod_anomaly_flowCbeta}, for a suitable open range of the real parameters $C, \gamma,\sigma \in \bR$.

\begin{thm}\label{thm: monotonicity funct}
Given an oriented, compact spin $7$-manifold $M$, consider a one-parameter family of $\rG_2$-structures $\varphi_t \in \Omega^3$ and smooth dilaton functions $\phi_t \in C^\infty(M,\bR)$ satisfying the
generic heterotic $\rG_2$ flow \eqref{eq: mod_anomaly_flowCbeta}, for parameters $C, \gamma,\sigma \in \bR$. Assuming the initial condition $d\left(e^{-4\phi_0}\psi_0\right) = 0$, one has
$$
\frac{\partial}{\partial t}\mathcal{M}(\varphi_t,\phi_t) =\int_M \left(-(3\gamma + 3) |d\phi|^2 + |\tau_3|^2 + 3\left(\frac{7(21 C - 2)}{144} - \sigma\right)\tau_0^2 \right)\operatorname{Vol}_\varphi 
$$
for all $t$. Consequently, provided that 
$$
\gamma < -1, \qquad \sigma < 
 \frac{7(21 C - 2)}{144},
$$
one has $\frac{\partial}{\partial t}\mathcal{M}(\varphi_t,\phi_t) \geqslant 0$, with equality only if $\varphi_t$ is torsion-free and $\phi_t$ is constant.
\end{thm}

\begin{proof}
Consider $(\tilde \varphi_t,\tilde \phi_t)$ solving \ref{eq: mod_anomaly_flowCbeta}, for parameters $C, \gamma,\sigma \in \bR$, and denote 
$$
(\varphi_t,\phi_t) = (f_t^*\tilde \varphi_t,f_t^*\tilde \phi_t),
$$ 
where $f_t$ is a one parameter family of diffeomorphisms generated by $\nabla e^{ 4\tilde \phi_t}$. Then, assuming that the initial condition $(\tilde \varphi_0,\tilde \phi_0)$ satisfies $d\left(e^{-4\tilde \phi_0}\tilde \psi_0\right) = 0$, we have that $e^{-4\phi_t}\psi_t$ is closed for all $t$ and
$$
(\varphi_t,\phi_t) \in \mathcal{C}_\tau, \qquad \tau = [e^{-4\phi_0}\psi_0] \in H^4(M,\mathbb{R}).
$$
Furthermore, it satisfies 
\begin{align}\label{eq:flowgaugefixed}
    \frac{\partial}{\partial t}\left(e^{-4\phi}\psi\right)=&-dH_{\varphi} + \frac{7}{4}Cd\left(\tau_0\varphi \right) - \mathcal{L}_{\nabla e^{4\phi}}(e^{-4\phi}\psi)\\\nonumber
    =&\ d\left(-H_{\varphi} + \frac{7}{4}C\tau_0\varphi - 4 \nabla \phi \lrcorner  \psi\right)\\
    \nonumber
    \frac{\partial}{\partial t}\phi =&\  e^{4\phi} \left(\Delta \phi - (\gamma + 4) |d\phi|^2 + \frac{1}{4}|\tau_3|^2 - \sigma \tau_0^2 \right).
\end{align}
Following the notation in Lemma \ref{lem:variation}, we have $\partial_t(e^{-4\phi_t}\psi_t) = dh$,
with
$$
h = - H_{\varphi} + \frac{7}{4}C\left(\tau_0\varphi \right) - 4\nabla \phi \lrcorner \psi = \star_\varphi d \varphi + \frac{7}{4}\left(C - \frac{2}{3}\right)\tau_0 \varphi,
$$
where we omit the $t$ subscript for simplicity in the notation. Consequently, along the flow \eqref{eq: mod_anomaly_flowCbeta}, one has
\begin{align*}  
\frac{\partial}{\partial t}\mathcal{M}(\varphi_t,\phi_t) & = 3\int_M \left(- (\gamma + 4) |d\phi|^2 + \frac{1}{4}|\tau_3|^2 - \sigma \tau_0^2 \right)\operatorname{Vol}_\varphi\\\
&\quad  + \frac{1}{4}\int_M  \left( \star_\varphi d \varphi + \frac{7}{4}\left(C - \frac{2}{3}\right)\tau_0 \varphi \right) \wedge d\varphi\\
& = \int_M \left(-3(\gamma + 4) |d\phi|^2 + \frac{3}{4}|\tau_3|^2 + \left(\frac{49}{16}\left(C - \frac{2}{3}\right) - 3\sigma \right)\tau_0^2 \right)\operatorname{Vol}_\varphi + \frac{1}{4}\int_M |d \varphi|^2\operatorname{Vol}_\varphi\\
& =\int_M \left(-(3\gamma + 3) |d\phi|^2 + |\tau_3|^2 + \left(\frac{49}{16}\left(C - \frac{2}{3}\right) - 3\sigma + \frac{7}{4}\right)\tau_0^2 \right)\operatorname{Vol}_\varphi,
\end{align*}
where we have used that $d\varphi=\tau_0\psi+3\tau_1\wedge \varphi+\star\tau_3$ and
$$
|d\varphi|^2 = 7 \tau_0^2 + 36|d\phi|^2 + |\tau_3|^2, \qquad |\psi|^2 \operatorname{Vol}_\varphi = \psi \wedge \varphi = 7 \operatorname{Vol}_\varphi.
$$
The proof follows from the fact the $\mathcal{M}$ is invariant under diffeomorphisms.
\end{proof}

With these preliminaries in place we can prove the second main result of this section, which characterizes fixed points of the generic heterotic $\rG_2$ flow \ref{eq: mod_anomaly_flowCbeta}, for an open range of the real parameters $C, \gamma,\sigma \in \bR$.


\begin{prop}\label{prop:weakG2functional}
Let $(\varphi,\phi)$ be a fixed point of the flow \eqref{eq: mod_anomaly_flowCbeta} on a compact manifold. That is, $(\varphi, \phi)$ solve the equations 
\begin{gather}\label{eq: mod_anomaly_flowCbetafixed}
\begin{split}
-dH_{\varphi}+\frac{7}{4}C d\left(\tau_0\varphi \right) & = 0,\\
\Delta \phi - \gamma |d\phi|^2 + \frac{1}{4}|\tau_3|^2 - \sigma \tau_0^2 &= 0.
\end{split}
\end{gather}
Assuming that
$$
\gamma < -1, \qquad \sigma < 
 \frac{7(21 C - 2)}{144},
$$
it follows that $\varphi$ is torsion-free and $\phi$ is constant.
\end{prop}

\section{Short-time existence}
\subsection{Conformal rescalings}

Let $M$ be an oriented, spin manifold of dimension seven.  In this section we collect some formulae for conformally rescaled versions of the flow \eqref{eq: mod_anomaly_flowCbeta} which we will use for the proof of well-posedness. We start recalling some formulae from conformal Riemannian/$\rG_2$-geometry.

\begin{lemma}\cite{KarigiannisFlowsI}\label{lm: conformal identities}
Let $\tvarphi$ be a $\rG_2$-structure and $\phi\in C^\infty(M,\bR)$. The $\rG_2$-structure $\varphi=e^{3\phi}\tvarphi$ has the following Riemannian and $\rG_2$ invariants:
\begin{align*}
\psi & =e^{4\phi}\tpsi,\\
g & =e^{2\phi}\tg,\\
R_{g} & =e^{-2\phi}(R_{\tg}-12\Delta_{\tg}\phi-30|d\phi|_{\tg}^2), \\
d^{\star}\alpha & =e^{-2\phi}(d^{\tilde{\star}}\alpha-(7-2k)\tilde{\nabla}\phi\lrcorner\alpha) \qforq \alpha\in \Omega^k.
    \end{align*}
Furthermore, one has the following formulae for the torsion classes
\begin{align*}
\tau_0=e^{-\phi}\tilde{\tau}_0, \qquad \tau_1=\tilde{\tau}_1+d\phi, \qquad \tau_2=e^{\phi}\tilde{\tau}_2, \qquad \tau_3=e^{2\phi}\tilde{\tau}_3,
\end{align*} 
and the characteristic torsion
\begin{align*}
H_{\varphi}=e^{2\phi}(H_{\tvarphi} - d\phi\lrcorner\tpsi).
\end{align*} 
\end{lemma}

\noindent Using this, we show a technical lemma necessary for the short-time existence proof.

\begin{lemma}\label{lm: modified anomaly flow conformal change}
Let $\tvarphi$ be a coclosed $\rG_2$-structure and $\phi$ be a smooth $\bR$-valued function on $M$. Then, for any constant $K\in \mathbb{R}$ the integrable $\rG_2$-structure $\varphi=e^{3\phi}\tvarphi$ satisfies
    \begin{align*}
    d\left(H_{\varphi}-K(\tr T)\varphi\right) & = d\left(e^{2\phi}\left(H_{\tvarphi}-d\phi\lrcorner\tpsi-K(\tr \tilde{T})\tvarphi\right)\right).
\end{align*}
\end{lemma}

\begin{proof}
From Lemma \ref{lm: conformal identities}, we have $H_{\varphi}-K(\tr T)\varphi =e^{2\phi}\left(H_{\tvarphi}-d\phi\lrcorner\tpsi-K(\tr \tilde{T})\tvarphi\right)$. 
\end{proof}

In the next result, we obtain the desired characterization of flow lines for \eqref{eq: mod_anomaly_flowCbeta} in terms of the coclosed $\rG_2$-structure $\tvarphi = e^{-3\phi} \varphi$.

\begin{lemma}\label{lem:confCbeta}
Given an oriented, spin $7$-manifold $M$, let $(\varphi_t,\phi_t)$ be a solution of the generic heterotic $\rG_2$ flow \eqref{eq: mod_anomaly_flowCbeta} with conformally coclosed initial condition $\varphi_0$. Then, setting $\tvarphi_t = e^{-3\phi_t} \varphi_t$ we have $\tpsi_t=e^{-4\phi_t}\psi_t$, and the family $(\tvarphi_t,\phi_t)$ is a solution of the coupled flow
\begin{align}\label{eq: modified anomaly flow_coclosed}
    \frac{\partial}{\partial t}\tpsi=&\ -d\left(e^{2\phi}\left(H_{\tvarphi}-d\phi\lrcorner\tpsi - \frac{7}{4}C\tilde{\tau}_0\tvarphi\right)\right),\\ \nonumber
    \frac{\partial}{\partial t}\phi =&\  e^{2\phi}\left(\Delta_{\tg}\phi + (5-\gamma)|d\phi|_{\tg}^2 + \frac{1}{4}|\ttau_3|^2_{\tg} -\sigma\tilde{\tau}_0^2\right),
\end{align}
such that $d\tpsi = 0$ for all $t$. Conversely, given $(\tvarphi_t,\phi_t)$ a solution of \eqref{eq: modified anomaly flow_coclosed}, where $\tvarphi_t$ is a one-parameter family of coclosed $\rG_2$-structures, the family $(\varphi_t,\phi_t) = (e^{3\phi_t}\tvarphi_t,\phi_t)$ satisfies the generic heterotic $\rG_2$ flow \eqref{eq: mod_anomaly_flowCbeta}.
\end{lemma}

\begin{proof}
The proof for the first and second evolution equation follow from Lemma \ref{lm: modified anomaly flow conformal change} and Lemma \ref{lm: conformal identities}, respectively.
\end{proof}

\subsection{Well-posedness via Nash-Moser}

In this section we prove the well-posedness of the generic heterotic $\rG_2$ flow \eqref{eq: mod_anomaly_flowCbeta}, for $C < - \tfrac{1}{3}$, which includes the value $C = -4/3$, of our main interest. For this, we will follow Hamilton's argument for the well-posedness of Ricci flow \cite{Hamilton3folds}, via the Nash–Moser Inverse Function Theorem \cite{Hamilton3folds}. We observe that the a similar argument has been used to prove short-time existence of the Laplacian flow of closed $\rG_2$-structures \cite{bryant2011laplacian}.

We recall first the fundamental result by Hamilton which we will apply. Let $M$ be an oriented compact manifold, $F$ a vector bundle over $M$, $U$ an open subbundle of $F$ and  
$$
E\colon \Gamma(M,U)\to  \Gamma(M,F)
$$
a second order differential operator. For $s\in \Gamma(M,U)$, we denote by 
$$
D_sE\colon \Gamma(M,F)\to \Gamma(M,F)
$$ 
the linearization of $E$ at $s$ and by 
$\sigma(D_sE)$ the principal symbol of
$D_sE$.

\begin{defn}
An \emph{integrability condition} for $E$ is a first order linear differential operator 
$$ 
L\colon \Gamma(M,F)\to \Gamma(M,G)\,,
$$
where $G$ is another vector bundle over $M$, such that $L \circ E = 0$. 
\end{defn}

Hamilton's abstract existence result can be stated as follows.

\begin{thm}{\cite[Theorem 5.1]{Hamilton3folds}}\label{thm:Ham}
Assume that $E$ admits an integrability condition such that all the eigenvalues of $\sigma(D_sE)$ restricted to $\ker \sigma(L)$ have strictly positive real part.  Then for every $s_0\in C^{\infty}(M,U)$ the flow 
\begin{equation}\label{flow_Ham}
\frac{\partial s}{\partial t}=E(s),\quad s(0) = s_0,
\end{equation}
is well-posed. 
\end{thm}

We will apply the previous theorem to the flow of coclosed $\rG_2$-structures \eqref{eq: modified anomaly flow_coclosed}, and then apply Lemma \ref{lem:confCbeta} to prove short-time existence for \eqref{eq: mod_anomaly_flowCbeta}. Upon fixing an orientation on our seven-dimensional manifold $M$, the map
$$
\Omega^3_+ \to \Omega^4: \widetilde{\varphi} \mapsto \widetilde{\psi}:= \star_{\widetilde{\varphi}} \widetilde{\varphi},
$$
where $\Omega^3_+ \subset \Omega^3$ denotes the open subset of $\rG_2$-structures, is open and invertible over its image, which we denote by $\Omega^4_+$. Consider the operator 
$$
P:C^\infty(M,\bR)\times\Omega^4_+\rightarrow C^\infty(M,\bR)\times\Omega^4
$$ 
given by $P = (P_1,P_2)$, where
\begin{gather}\label{eq: operator P}
\begin{split}
P_1(\phi,\widetilde{\psi}) & = e^{2\phi}\left(\Delta_{\tg}\phi + (5-\gamma)|d\phi|_{\tg}^2 + \frac{1}{4}|\ttau_3|^2_{\tg} -\sigma\tilde{\tau}_0^2\right),\\
P_2(\phi,\widetilde{\psi}) & = -d\left(e^{2\phi}\left(H_{\tvarphi}-d\phi\lrcorner\tpsi - \frac{7}{4}C\tilde{\tau}_0\tvarphi\right)\right).
\end{split}
\end{gather}

\begin{prop}
Let $(\phi,\widetilde{\psi}) \in C^\infty(M,\bR)\times\Omega^4_+$ such that $d\widetilde{\psi} = 0$. Then, the linearization of \eqref{eq: operator P} at $(\phi,\tpsi)$ is
$$
D_{(\phi,\tpsi)}P(f,\chi)=e^{2\phi}\left( a,\star_{\tvarphi}(Y\lrcorner\tpsi+S\widetilde{\diamond} \tvarphi)\right)
$$ 
where
\begin{align}\label{eq: lin_function}
a=&\ \Delta_{\tg}f+\mathrm{L.O.T.},\\ \label{eq: lin_vector} 
         Y=&\ -(C+\frac13)d(\Div_{\tg} X)+ \mathrm{L.O.T.},\\ \label{eq: lin_symmetric}
S=&\ \Delta_{\tg}h-\frac12\cL_{\curl_{\tvarphi} X+\Div_{\tg} h-\frac14\widetilde{\nabla}(\tr h)-\widetilde{\nabla} f}\tg\\ \nonumber
& \ -\frac13\left(\frac14\Delta_{\tg}(\tr h)-\Div(\Div h)-\Delta_{\tg}f\right)\tg+ \mathrm{L.O.T.},
\end{align}
for arbitrary $f \in C^\infty(M,\bR)$ and $\chi=*(X\lrcorner\tpsi+h\widetilde{\diamond} \tvarphi) \in \Omega^4$.
\end{prop}

\begin{proof}
By definition of $P_1$, its linearization at $(\phi,\tpsi)$ is 
$$
D_{(\phi,\tpsi)}P_1(f,\chi)=e^{2\phi}\Delta_{\tg}f+\mathrm{L.O.T.}.
$$
Now, using that $d \tpsi = 0$, we have
   \begin{align*}
       P_2(\phi,\tvarphi)=-e^{2\phi}d\left(H_{\tvarphi}-\frac74C\ttau_0\tvarphi\right)+e^{2\phi}\cL_{\tnabla\phi}\tpsi.
   \end{align*}
   From the decomposition \eqref{eq: components dH} of $dH_{\tvarphi}$, for $D_{\tpsi}(dH_{\tvarphi})(\chi)=W\wedge\tvarphi+R\diamond\tpsi$ we have 
    \begin{align*}
        W_a=-\frac13 \tnabla_a(\tr \Dot{\tT})+\lot \qandq R_{ab}=-\frac12\curl_{\tvarphi}\Dot{T}_{ab}-\frac12\curl_{\tvarphi}\Dot{T}_{ba}+\lot
    \end{align*}
By \eqref{eq: ddt T}, we get
    \begin{align}\label{eq: dot nabla_trT}
        \tnabla_a(\tr \Dot{\tT})=&\tnabla_a(\Div_{\tg} X)+\lot
        \end{align}
        and
        \begin{align*}
    \curl_{\tvarphi}\Dot{\tT}_{ab}=&\left(\tnabla_i\tnabla_k h_{al}\tvarphi_{klj}+\tnabla_i\tnabla_aX_j-\frac14\tnabla_i\tnabla_k(\tr h)\tvarphi_{kaj}\right)\tvarphi_{ijb}+\lot\\ \nonumber
        =&\ \tnabla_i\tnabla_bh_{ai}-\tnabla_i\tnabla_ih_{ab}+\tnabla_a(\curl_{\tvarphi} X)_b-\frac{1}{4}\tnabla_a\tnabla_b(\tr h)+\frac{1}{4}\tnabla_i\tnabla_i(\tr h)\tg_{ab}+\lot\\
         =&\ \tnabla_b(\Div_{\tg} h)_a-\Delta_{\tg}h_{ab}+\tnabla_a(\curl_{\tvarphi} X)_b-\frac{1}{4}\tnabla_a\tnabla_b(\tr h)+\frac{1}{4}\Delta_{\tg}(\tr h)\tg_{ab}+\lot,
    \end{align*}
where we used the contraction \eqref{eq: varphi1 varphi1} and the Ricci identity on $h$ and $X$.  By combining these identities with equation (\ref{eq: components dH}) for $d H$ and the expression of $\Ric$ in terms of the torsion in Lemma \ref{lm: Ricci and scalar curv}, we obtain
\begin{align*}
    D_{(\phi,\tpsi)}&\left(dH_{\tvarphi}\right)(f,\chi)\\
    =&\ \left(\Delta_{\tg}h-\frac14\Delta_{\tg}(\tr h)\tg-\frac12\cL_{\Div_{\tg}h+\curl_{\tvarphi}X-\frac14\tnabla(\tr h)}\tg\right)\diamond\psi -\frac13d(\Div_{\tg}X)\wedge\tvarphi+\lot \\
    =& \ \star\left(-\left(\Delta_{\tg}h-\frac12\cL_{\Div_{\tg}h+\curl_{\tvarphi}X-\frac14\tnabla(\tr h)}\tg-\frac{1}{3}\left(\Delta_{\tg}(\tr h)-\Div_{\tg}(\Div_{\tg}h)\right)\tg\right)\diamond\tvarphi\right.\\
    & \ \left. +\frac13d(\Div_{\tg}X)\lrcorner\tpsi\right)+\lot.
\end{align*}
Furthermore, it follows by \eqref{eq: dot nabla_trT} that
\begin{align*}
D_{(\phi,\tpsi)}\left(\frac74 d(\tau_0\tvarphi)\right)(f,\chi)=&-d(\Div_{\tg} X)\wedge\tvarphi+\mathrm{L.O.T.}.
\end{align*}
We recall that for an arbitrary vector field $Z$ in $M$, we have
\cite[(4.40)]{Grigorian2013}
\begin{align*}
\cL_Z\tpsi=\star\left(\frac{1}{2}(\curl_{\tvarphi} Z)\lrcorner\tpsi+\left(\frac13(\Div_{\tg} Z)\tg-\cL_Z \tg\right)\widetilde{\diamond} \tvarphi\right)+ \mathrm{L.O.T.}
\end{align*}
Then, it follows that
\begin{align*}
D_{(\phi,\tpsi)}\left(\cL_{\tnabla \phi}\tpsi\right)(f,\chi)=&\star\left(\left(\frac13(\Div_{\tg} {df})\tg-\cL_{\tnabla f} \tg\right)\widetilde{\diamond} \tvarphi\right)+\mathrm{L.O.T.}.
\end{align*}
Combining these, the proof follows.
\end{proof}

In order to apply Theorem \ref{thm:Ham}, we fix $(\phi_0,\tpsi_0) \in C^\infty(M,\bR)\times\Omega^4_+$ such that $d \tpsi_0 = 0$. Given $(\phi,\tpsi) \in C^\infty(M,\bR)\times\Omega^4_+$, we denote
$$
(\phi,\tpsi) = (\phi_0 + f,\tpsi_0 + \chi)
$$
for $(f,\chi) \in C^\infty(M,\bR)\times\Omega^4$, and define a vector field 
$$
Z_{(\phi,\tpsi)}=e^{2\phi}\left(-2\curl_{\tvarphi} X+\frac34\widetilde{\nabla}(\tr_{\tg} h)-\widetilde{\nabla} f\right),
$$
where $\chi=\star_{\tvarphi}(h\tilde{\diamond}\tvarphi+X\lrcorner\tpsi)$. Consider the operator 
$$
E:C^\infty(M,\bR)\times\Omega^4_+\rightarrow C^\infty(M,\bR)\times\Omega^4
$$ 
given by $E = (E_1,E_2)$, where
$$
E(\phi,\tpsi) =\left(P_1(\phi,\tpsi)+\cL_{Z_{(\phi,\tpsi)}}\phi,P_2(\phi,\tpsi) + \cL_{Z_{(\phi,\tpsi)}}\tpsi\right).
$$
Notice that 
\begin{equation*}
    E_1(\phi,\tpsi)=P_1(\phi,\tpsi)+\cL_{Z_{(\phi,\tpsi)}}\phi=P_1(\phi,\tpsi)+\tg(\tnabla\phi,Z_{(\phi,\tpsi)}),
\end{equation*}
then
\begin{equation*}
D_{(\phi,\tpsi)}E_1(f,\chi)=D_{(\phi,\tpsi)}P_1(f,\chi)+\lot.
\end{equation*}
If the $4$-form $\chi$ is closed, the vanishing condition on $\pi_7(d\chi)\in\Omega^5_7\simeq\Omega^1$ is equivalent to \cite[(2.24b)]{Grigorian2013}
\begin{equation}\label{eq: pi_7 d_chi}
    \curl_{\tvarphi} X-\Div_{\tg} h-\frac12\tnabla(\tr h)+\lot=0.
\end{equation}
Thus, the symmetric part of the linearization of $P_2$ above (see \eqref{eq: lin_symmetric}) becomes
\begin{align*}
     S
        =&\ \cL_{\curl_{\tvarphi}X}\tg-\frac38\cL_{\tnabla(\tr h)}\tg-\Delta_{\tg}h+\frac14\left(\Delta_{\tg}(\tr h)-\Div_{\tg}(\curl_{\tvarphi} X)\right)\tg+ \lot.
\end{align*}
Since 
 $$
 \curl_{\tvarphi}(\curl_{\tvarphi} X)=d(\Div_{\tg} X)-\Delta_{\tg}X+\lot,
 $$
 using \eqref{eq: lin_vector} and \eqref{eq: lin_symmetric}, we have 
$$
D_{(\phi,\tpsi)}E_2(f,\chi)=e^{2\phi}\star(\tilde{Y}\lrcorner\tpsi+\tilde{S}\widetilde{\diamond} \tvarphi),
$$ 
where
\begin{align}\label{eq: tilde_Y and tilde_S}
\begin{split}
   \tilde{Y}=&\ -\left(C+\frac43\right)d(\Div_{\tg} X)+\Delta_{\tg}X+  \mathrm{L.O.T.}\\
    \tilde{S}=&\ \Delta_{\tg}h+ \mathrm{L.O.T.}
    \end{split}
\end{align}

\begin{lemma}\label{lem: principal symbol}
Assume that $(\phi,\tpsi) \in C^\infty(M,\bR)\times\Omega^4_+$ satisfies $d\tpsi = 0$. Then, for every $x\in M$ and $\xi\in T_xM\backslash \{0\}$, the principal symbol of $D_{(\phi,\tpsi)}E$ is
    \begin{align}\label{eq: principal_symbol}
\sigma_\xi\left(D_{(\phi,\tpsi)}E\right)(f,\chi)=e^{2\phi}\left(|\xi|^2_{\tg}f,\star_{\tvarphi}\left(-\left(C+\frac43\right)\langle X,\xi\rangle_{\tg} (\xi\lrcorner\tpsi)+|\xi|^ 2_{\tg}\star_{\tvarphi}\chi\right)\right).
    \end{align}
    Moreover, 
\begin{equation*}
    \Big{\langle}  \sigma_\xi\left(D_{(\phi,\tpsi)}E\right)(f,\chi), (f,\chi)\Big{\rangle} >0 \quad \text{if and only if} \quad C<-\frac13.
\end{equation*}
\end{lemma}

\begin{proof}
The expression \eqref{eq: principal_symbol} follows from \eqref{eq: tilde_Y and tilde_S}. And taking the inner product of \eqref{eq: principal_symbol} with $(f,\chi)$, we have:
    \begin{align*}
  e^{-2\phi} \Big{\langle}  \sigma_\xi\left(D_{(\phi,\tpsi)}E\right)(f,\chi), (f,\chi)\rangle & =|\xi|^2f^2-4\left(C+\frac43\right)\Big{\langle} X,\xi\rangle^2+4|\xi|^2|X|^2+|\xi|^2|h\widetilde{\diamond}\tvarphi|^2 \\
   & \geq 4 \left(1 - C - \frac43 \right)|\xi|^2|X|^2+|\xi|^2|h\diamond\varphi|^2,
\end{align*}
where all the products are calculated with respect to $\tg$.
\end{proof}

With the preliminary lemmas in place we are ready to establish the short-time existence. We sketch the proof as it follows standard lines.
\begin{thm}\label{thm:STE}
Let $(\varphi_0,\phi_0) \in C^\infty(M,\bR)\times\Omega^3_+$, such that $d(e^{-4\phi_0} \psi_0) = 0$. Then, for any $C<-\frac13$ and $\gamma, \sigma\in \bR$, there exists $\epsilon>0$, depending on $(\varphi_0,\phi_0)$, and a unique solution $(\varphi_t,\phi_t)$ of the generic heterotic $\rG_2$ flow \eqref{eq: mod_anomaly_flowCbeta} defined for $t \in [0,\epsilon)$, with initial condition $(\varphi_0, \phi_0)$, such that
$$
d(e^{-4\phi_t} \psi_t) = 0
$$ 
for all $t\in [0,\epsilon)$.
\begin{proof}
    Following the notation in Theorem \ref{thm:Ham}, we consider the linear operator
$$
L \colon C^\infty(M,\bR)\times\Omega^4 \to \Omega^5 \colon (\phi,\chi) \to d\chi. 
$$
By Lemma \ref{lem: principal symbol}, $L$ is an integrability condition for the operator $E$ and, by direct application of Theorem \ref{thm:Ham}, the flow
\begin{align*}
     \frac{\partial}{\partial t}\tpsi & = -d\left(e^{2\phi}\left(H_{\tvarphi}-d\phi\lrcorner\tpsi - \frac{7}{4}C\tilde{\tau}_0\tvarphi\right)\right)  + \cL_{Z_{(\phi,\tpsi)}}\tpsi \\ \nonumber
    \frac{\partial}{\partial t}\phi& =  e^{2\phi}\left(\Delta_{\tg}\phi + (5-\gamma)|d\phi|_{\tg}^2 + \frac{1}{4}|\ttau_3|^2_{\tg} -\sigma\tilde{\tau}_0^2\right),
\end{align*}
is well-posed for coclosed initial data $d \tpsi_0 = 0$, provided that $C < -1/3$. Pulling back the flow line by the diffeomorphism generated by $Z_{(\phi,\tpsi)}$ and applying Lemma \ref{lem:confCbeta}, we obtain the desired short-time existence for the generic heterotic $\rG_2$ flow \eqref{eq: mod_anomaly_flowCbeta} for conformally coclosed initial data.
\end{proof}
\end{thm}


\section{Evolution Equations and Shi-type Estimates}

\subsection{Preliminary estimates}\label{section: preliminary estimates}

In this section, we will give evolution equations and Shi-type estimates for the heterotic $\rG_2$ flow. For this, we first prove some preliminary estimates. We will consider a conformally coclosed $\rG_2$-structures with dilaton $(\varphi_t, \phi_t)$, so $d(e^{-4\phi_t} \psi_t) = 0$. Setting $\varphi_t = e^{3 \phi_t} \tilde{\varphi}_t$, where $\tilde{\varphi}_t = e^{-3\phi_t} \varphi_t$ is a coclosed $\rG_2$-structure, we have that  $\tau_1 = d \phi$ by~\eqref{eq:d:varpsi} and Lemma~\ref{lm: conformal identities}.  From~\cite[(2.11)]{DwivediGianniotisKarigiannis}, there is a constant $c$ such that $\nabla_k \phi = c(d \phi \lrcorner \varphi)_{ab} \varphi_{abk}$. Since $T$ includes the term $\tau_1\lrcorner \varphi = d\phi \lrcorner \varphi$, we have that
\begin{align}\label{eqn:nablaDil1}
    \nabla \phi = T \ast \varphi.
\end{align}
Our star notation includes any metric contraction of the indicated tensors. Since $\nabla \varphi = T \ast \psi$, we have that
\begin{align}\label{eqn:nablaDil2}
    \nabla^2 \phi = \nabla T \ast \varphi + T^2 \ast \psi,
\end{align}
where $T^2$ denotes $T \ast T$. 

We now introduce notation to represent polynomials in tensors comprised of derivatives of $\Rm$ and $T$, together with $\varphi$ and $\psi$.  For $k \geq 1$, we use the notation $P(k)$ to denote a tensor which is polynomial in $\nabla^i T$, $T^l$, $\nabla^j \Rm$, and $\Rm^m$, specifically a tensor formed by a sum of terms of the form 
\begin{equation}\label{eqn:P(Rm, T, k) star}\nabla^{j_1}\Rm \ast \cdots \ast \nabla^{j_q}\Rm \ast \Rm^m \ast \nabla^{i_1} T \ast \cdots \nabla^{i_p} T \ast T^{l},\end{equation}
where
\begin{align}\label{eqn:ksumPRmTk}
\left(i_1 + \cdots + i_p + p + l\right) + \left(j_1 + \cdots + j_q + 2q + 2m\right) = k.\end{align}
Our star notation in ~\eqref{eqn:P(Rm, T, k) star} includes any metric contraction of the indicated tensors as well as terms of the form $\ast \varphi^a \ast \psi^b$. We suppress terms of the form $\ast \varphi^a \ast \psi^b$ from our star notation, but their only function, for the purposes of our estimates, is to produce additional torsion terms when differentiating due to~\eqref{eqn:nablaDil1} and~\eqref{eqn:nablaDil2}. 

Using ~\eqref{eqn:nablaDil1} and ~\eqref{eqn:nablaDil2}, together with $\nabla \psi = T\ast \varphi$,
\begin{align}\label{eqn:DerivativesDilaton}
    \nabla^k \phi &= P(k),\\
    \nabla^k \varphi &= P(k),\\
    \nabla^k \psi &= P(k).\label{eqn:Derivativesphipsi}
\end{align}
It follows that
\begin{align}\label{eqn:DerivativePNotation}
    \nabla^{k_1} P(k_2) &= P(k_1+k_2),\\\label{eqn:MultiplicativePNotation}
    P(k_1) \ast P(k_2) &= P(k_1 + k_2).
\end{align}
Combining~\eqref{eqn:DerivativesDilaton} with ~\eqref{eqn:DerivativePNotation} and~\eqref{eqn:MultiplicativePNotation}, 
\begin{align}\label{eqn:ePhiPNotation}
\nabla^{k_1}\left(e^{4\phi}P(k_2)\right) = e^{4\phi} P(k_1 + k_2).
\end{align}
We now give bounds on $P(k)$ which are used in our Shi-type estimates. 

\begin{lemma}\label{lemma norm PRmTk}
For each $k \geq 3$,
\begin{align}\label{eqn:normPRmTk}
    |P(k)| \leq C\left(|T|^k + \left|\Rm\right|^{\frac{k}{2}} + \sum_{m=1}^{k-1} |\nabla^{m}T|^{\frac{k}{m+1}} + \sum_{m=1}^{k-2} \left|\nabla^m \Rm\right|^{\frac{k}{m+2}}\right).
\end{align}
If $P(k)$ does not include any terms with $\nabla^{j_1}\Rm$ or $\nabla^{i_1}T$, then the corresponding term can be omitted from the sum on the right-hand side of~\eqref{eqn:normPRmTk}.
\end{lemma}
\begin{proof}
The generalized Young's inequality says $|x_1|\cdots |x_m| \leq |x_1|^{a_1}/a_1 + \cdots + |x_m|^{a_m} / a_m$, where $1/a_1 +\cdots + 1/a_m = 1$. Applying this inequality,
\begin{gather} \label{eqn:generalized Young's}
\begin{split}
    |&\nabla^{j_1}\Rm| \cdots |\nabla^{j_q}\Rm| \left|\Rm\right|^m |\nabla^{i_1} T|  \cdots |\nabla^{i_p} T| |T|^{l}\\
    &\leq C\left(|\nabla^{j_1}\Rm|^{\frac{k}{j_1+2}} + \cdots + |\nabla^{j_q}\Rm|^{\frac{k}{j_q + 2}} \right.\\
    &\ \left. \qquad + \left|\Rm\right|^{\frac{k}{2}} + |\nabla^{i_1} T|^{\frac{k}{i_1 + 1}} + \cdots + |\nabla^{i_p} T|^{\frac{k}{i_p+1}} + |T|^{k} \right),
\end{split}
\end{gather}
which follows since, by~\eqref{eqn:ksumPRmTk},
\begin{align*}
     \frac{j_1+2}{k} + \cdots + \frac{j_q +2}{k} + \frac{2m}{k} + \frac{i_1 +1}{k} + \cdots + \frac{i_p +1}{k} + \frac{l}{k} = \frac{k}{k} = 1.
\end{align*}
Using~\eqref{eqn:generalized Young's} and the definition of $P(k)$, we find that~\eqref{eqn:normPRmTk} holds. If a term of the form $\nabla^{j_1}\Rm$ or $\nabla^{i_1}T$ does not appear in $P(k)$, then it will not appear in ~\eqref{eqn:generalized Young's}, so it can be omitted from the sum on the right-hand side of ~\eqref{eqn:normPRmTk}.
\end{proof}

\subsection{Evolution of curvature for heterotic \texorpdfstring{$\rG_2$}{G2} flow}

For each $k \geq 0$, we compute the evolution of $\nabla^k \Rm$ and estimate the evolution of $\left|\nabla^k \Rm\right|^2$ along generic heterotic $\rG_2$ flow.  For emphasis, recall that our star notation suppresses any terms of the form $\ast \varphi^a \ast \psi^b$. 

\begin{prop}\label{proposition evolution nabla^k Rm} Let $(\varphi_t, \phi_t)$ be a solution to the generic heterotic $\rG_2$ flow ~\eqref{eq: mod_anomaly_flowCbeta} with initial condition $(\varphi_0, \phi_0)$ satisfying $d(e^{-4\phi_0} \psi_0)=0$.  Then, for each $k \geq 0$, 
\begin{align}\label{eqn:DerivkRm}
     \frac{\partial}{\partial t} \nabla^k \Rm = e^{4\phi} \left(\Delta \nabla^k \Rm +  T \ast \nabla^{k+1} \Rm + T \ast \nabla^{k+2}T +  P(k+4)\right),
\end{align}
Furthermore, the $P(k+4)$ in~\eqref{eqn:DerivkRm} does not include any terms with $\nabla^{k+2}\Rm$, $\nabla^{k+1}\Rm$, $\nabla^{k+3}T$, or $\nabla^{k+2}T$.
\end{prop}
\begin{proof}
Since $(\varphi_t, \phi_t)$ is a given solution to the generic-heterotic $\rG_2$ flow with conformally coclosed initial condition, we have that $d(e^{-4\phi_t} \psi_t)=0$ for as long as the flow exists. So, the estimates of Section \ref{section: preliminary estimates} hold.

By Proposition~\ref{p:evolutiongvarphiCbeta}, we have that 
\begin{align}\label{eq:RmEvolutionGeneral}
\frac{\partial g}{\partial t}=e^{4\phi}\left(-2\Ric_{g}-4\cL_{d\phi}g+2E\right),\end{align}
where we define $$E:= 
\frac14 H_{\varphi}^2 + \frac{7}{4}C\tau_0 T_\sym+\left(\left(\frac{7}{24}-\sigma\right)\tau_0^2+\left(3-\gamma\right)|d\phi|^2 \right)g.$$
Note that we have hidden the dependence on the constants $C,\gamma, \sigma$ in $E$. 

Now, if $\frac{\partial}{\partial t} g_{ij} = h_{ij}$, then the evolution of $\Rm$ is given by~\cite[Lemma 6.5]{KnopfBook}.
\begin{align}\label{eqn:standardevolutionRm}
    \frac{\partial}{\partial t}R_{ijk}^{\quad l} = \frac{1}{2}g^{lp} \left( \nabla_i \nabla_k h_{jp} + \nabla_j \nabla_p h_{ik} - \nabla_i \nabla_p h_{jk} - \nabla_j \nabla_k h_{ip}  - R_{ijk}^{\quad q}h_{qp} - R_{ijp}^{\quad q} h_{kq}\right).
\end{align}
\noindent Applying ~\eqref{eq:RmEvolutionGeneral} to~\eqref{eqn:standardevolutionRm} and using the Bianchi identity as in~\cite[Lemma 6.13]{KnopfBook},
\begin{align}\label{eqn:fullEvolutionRm}
    \frac{\partial}{\partial t} R_{ijk}^{\quad l} &= e^{4\phi}\Delta R_{ijk}^{\quad l} + e^{4\phi} g^{pq}\left(R_{ijp}^{\quad r}R_{rqk}^{\quad l} -2R_{pik}^{\quad r}R_{jqr}^{\quad l} + 2R_{pir}^{\quad l}R_{jqk}^{\quad r} \right)\nonumber\\
    & + e^{4\phi} \left( R_{p}^l R_{ijk}^{\quad p} - R_i^{p}R_{pjk}^{\quad l} - R_j^p R_{ipk}^{\quad l} - R_{k}^p R_{ijp}^{\quad l}\right)\nonumber\\
    & +4 e^{4\phi} g^{lp} \left(-\nabla^4_{ikjp} \phi - \nabla^4_{jpik} \phi + \nabla^4_{ipjk} \phi +\nabla^4_{jkip} \phi + R_{ijk}^{\quad q} \nabla_q \nabla_p \phi + R_{ijp}^{\quad q} \nabla_k \nabla_q \phi \right)\nonumber\\
    & + e^{4\phi} g^{lp} \left( \nabla_i \nabla_k E_{jp} + \nabla_j \nabla_p E_{ik} - \nabla_i \nabla_p E_{jk} - \nabla_j \nabla_k E_{ip}  - R_{ijk}^{\quad q} E_{qp} - R_{ijp}^{\quad q} E_{kq}\right)\nonumber\\
    & - 4 e^{4 \phi} g^{lp}\left(\nabla_i \phi \nabla_k R_{jp}+ \nabla_k \phi \nabla_i R_{jp} +\nabla_j \phi \nabla_p R_{ik} + \nabla_p \phi \nabla_j R_{ik} \right)\nonumber\\
    & + 4 e^{4 \phi} g^{lp}\left(\nabla_i \phi \nabla_p R_{jk}+ \nabla_p \phi \nabla_i R_{jk} +\nabla_j \phi \nabla_k R_{ip} + \nabla_k \phi \nabla_j R_{ip}\right)\nonumber\\
    & + 16 e^{4\phi} g^{lp}\left(-R_{jp} \nabla_i \nabla_k \phi - R_{ik} \nabla_j \nabla_p  \phi + R_{jk} \nabla_i \nabla_p  \phi + R_{ip} \nabla_j \nabla_k \phi\right)\nonumber\\
    & - 16 e^{4 \phi} g^{lp}\left(\nabla_i \phi \,\nabla^3_{kjp}\phi + \nabla_k \phi \,\nabla^3_{ijp}\phi +\nabla_j \phi\, \nabla^3_{pik}\phi + \nabla_p \phi \,\nabla^3_{jik} \phi \right)\nonumber\\
    & + 16 e^{4 \phi} g^{lp}\left(\nabla_i \phi \, \nabla^3_{pjk} \phi + \nabla_p \phi \, \nabla^3_{ijk} \phi + \nabla_j \phi \, \nabla^3_{kip} \phi + \nabla_k \phi \, \nabla^3_{jip} \phi \right)\nonumber\\
    & + 64 e^{4\phi} g^{lp}\left(-\nabla_i \nabla_k \phi \,\nabla_j \nabla_p \phi - \nabla_j \nabla_p  \phi \, \nabla_i \nabla_k \phi + \nabla_i \nabla_p  \phi \,\nabla_j \nabla_k \phi + \nabla_j \nabla_k \phi \, \nabla_i \nabla_p \phi \right)\nonumber\\
    & + 4 e^{4 \phi} g^{lp}\left(\nabla_i \phi \nabla_k E_{jp}+ \nabla_k \phi \nabla_i E_{jp} +\nabla_j \phi \nabla_p E_{ik} + \nabla_p \phi \nabla_j E_{ik} \right)\nonumber\\
    & - 4e^{4 \phi} g^{lp}\left(\nabla_i \phi \nabla_p E_{jk}+ \nabla_p \phi \nabla_i E_{jk} +\nabla_j \phi \nabla_k E_{ip} + \nabla_k \phi \nabla_j E_{ip}\right)\nonumber\\
    & + 4e^{4\phi} g^{lp}\left(-E_{jp}\nabla_i \nabla_k \phi - E_{ik} \nabla_j \nabla_p  \phi + E_{jk} \nabla_i \nabla_p  \phi  + E_{ip} \nabla_j \nabla_k \phi\right),\nonumber\\
\intertext{Applying the contracted second Bianchi identity, $\nabla_q R_{ijk}^{\quad q} = \nabla_i R_{jk}- \nabla_{j} R_{ik}$, using that $\nabla^4_{ipjk} \phi - \nabla^4_{ikjp} \phi = \nabla_i (\nabla^3_{pkj} - \nabla^3_{kpj} \phi) = -\nabla_i(R_{pkj}^{\quad q} \nabla_q \phi)$, and commuting covariant derivatives on all $\nabla^3$ terms,}
    &= e^{4\phi}\Delta R_{ijk}^{\quad l} + e^{4\phi} g^{pq}\left(R_{ijp}^{\quad r}R_{rqk}^{\quad l} -2R_{pik}^{\quad r}R_{jqr}^{\quad l} + 2R_{pir}^{\quad l}R_{jqk}^{\quad r} \right)\nonumber\\
    & + e^{4\phi} \left( R_{p}^l R_{ijk}^{\quad p} - R_i^{p}R_{pjk}^{\quad l} - R_j^p R_{ipk}^{\quad l} - R_{k}^p R_{ijp}^{\quad l}\right)\nonumber\\
    & +4 g^{lp} \left(-\nabla_i(R_{pkj}^{\quad q} \nabla_q \phi) -\nabla_j(R_{kpi}^{\quad q} \nabla_q \phi) + R_{ijk}^{\quad q} \nabla_q \nabla_p \phi + R_{ijp}^{\quad q} \nabla_k \nabla_q \phi \right)\nonumber\\
    & + e^{4\phi} g^{lp} \left( \nabla_i \nabla_k E_{jp} + \nabla_j \nabla_p E_{ik} - \nabla_i \nabla_p E_{jk} - \nabla_j \nabla_k E_{ip}  - R_{ijk}^{\quad q} E_{qp} - R_{ijp}^{\quad q} E_{kq}\right)\nonumber\\
    & + 4 e^{4 \phi} g^{lp}\left(\nabla_i  \phi \nabla_q R_{pkj}^{\quad q} + \nabla_p \phi \nabla_q R_{ijk}^{\quad q} + \nabla_j \phi \nabla_q R_{kpi}^{\quad q} + \nabla_k \phi \nabla_q R_{jip}^{\quad q}\right)\nonumber\\
    & -16 e^{4\phi} g^{lp}\left(\nabla_i \nabla_k \phi\, (R_{jp} + 4\nabla_{j}\nabla_{p}\phi) +\nabla_j \nabla_p \phi\,  (R_{ik} + 4\nabla_{i}\nabla_{k} \phi)\right)\\
    & + 16 e^{4\phi} g^{lp} \left(\nabla_i \nabla_p  \phi \, (R_{jk} + 4\nabla_{j}\nabla_{k}\phi) + \nabla_j \nabla_k \phi \,(R_{ip} +  4\nabla_{i} \nabla_{p}\phi )\right)\nonumber\\
    & - 16 e^{4 \phi} g^{lp}\left(\nabla_i \phi \, R_{pkj}^{\quad q}  + \nabla_p \phi \, R_{ijk}^{\quad q}  + \nabla_j \phi \, R_{kpi}^{\quad q}  + \nabla_k \phi \, R_{jip}^{\quad q} \right)\nabla_q \phi\nonumber\\
    & + 4 e^{4 \phi} g^{lp}\left(\nabla_i \phi \nabla_k E_{jp}+ \nabla_k \phi \nabla_i E_{jp} +\nabla_j \phi \nabla_p E_{ik} + \nabla_p \phi \nabla_j E_{ik} \right)\nonumber\\
    & - 4e^{4 \phi} g^{lp}\left(\nabla_i \phi \nabla_p E_{jk}+ \nabla_p \phi \nabla_i E_{jk} +\nabla_j \phi \nabla_k E_{ip} + \nabla_k \phi \nabla_j E_{ip}\right)\nonumber\\
    & + 4e^{4\phi} g^{lp}\left(-E_{jp}\nabla_i \nabla_k \phi - E_{ik} \nabla_j \nabla_p  \phi + E_{jk} \nabla_i \nabla_p  \phi  + E_{ip} \nabla_j \nabla_k \phi\right).\nonumber
\end{align}
Since $H^2_{\varphi}=T^2+T^2\ast\varphi+T^2*\ast\psi$, we can schematically write $E$ as 
\begin{align}\label{eqn:Eschematically}
    E = T\ast T \ast (1+ \varphi + \psi) = P(2).
\end{align}
Applying ~\eqref{eqn:nablaDil1}, ~\eqref{eqn:nablaDil2}, and~\eqref{eqn:Eschematically} to the evolution of $R_{ijk}^{\quad l}$ ~\eqref{eqn:fullEvolutionRm},
\begin{align}\label{eqn:P4inDtRm}
     \frac{\partial}{\partial t} \Rm &= e^{4\phi} \Big(\Delta \Rm+ \nabla \Rm \ast T  + T \ast \nabla^2 T + P(4)\Big).
\end{align}
Analyzing ~\eqref{eqn:fullEvolutionRm} and $\nabla^i E$, for $i \leq 2$, we find that the $P(4)$ in \eqref{eqn:P4inDtRm} does not include any terms with $\nabla^2 \Rm$, $\nabla \Rm$, $\nabla^3 T$, or $\nabla^2 T$.  This proves Proposition \ref{proposition evolution nabla^k Rm} with $k=0$.

We now write the evolution of $\nabla^k \Rm$ schematically for each $k>0$.  Using~\eqref{eqn:nablaDil2} and~\eqref{eqn:Eschematically}, we have that
\begin{align}\label{eqn:dgdtschematically}
    \frac{\partial}{\partial t} g = e^{4\phi} P(2).
\end{align}    
Applying the commutator of the time derivative with $\nabla^k$ (see, for instance, ~\cite[4.18]{LotayWeiLaplacianFlow}) and then using~\eqref{eqn:P4inDtRm} and~\eqref{eqn:dgdtschematically},
\begin{align}
    \frac{\partial}{\partial t} \nabla^k \Rm &= \nabla^k \left(\frac{\partial}{\partial t} \Rm\right) + \sum_{i=1}^k \nabla^{k-i} \Rm \ast \nabla^i \frac{\partial g}{\partial t}\nonumber\\
    &= \nabla^k \left(e^{4\phi} \Big(\Delta \Rm+ \nabla \Rm \ast T + T\ast \nabla^2 T + P(4)\Big)\right) + \sum_{i=1}^k \nabla^{k-i} \Rm \ast \nabla^i \left(e^{4\phi} P(2)\right).\nonumber\\
\intertext{Applying~\eqref{eqn:DerivativePNotation}, ~\eqref{eqn:MultiplicativePNotation}, and ~\eqref{eqn:ePhiPNotation},}
    &= \nabla^k \left(e^{4\phi} \Big(\Delta \Rm+ \nabla \Rm \ast T + T\ast \nabla^2 T\Big)\right) + e^{4\phi} P(k+4),\nonumber\\
\intertext{where the $P(k+4)$ includes no terms with $\nabla^{k+2}\Rm$, $\nabla^{k+1}\Rm$, $\nabla^{k+3}T$, or $\nabla^{k+2}T$. Expanding the first term and then applying the commutator of $\nabla^k$ and $\Delta$ (see~\cite[Section 7.2]{KnopfBook}),}
    &= e^{4\phi} \sum_{i=0}^k \nabla^{k-i} \phi \ast \nabla^i \Delta \Rm + \nabla^k \left(e^{4\phi} \Big(\nabla \Rm \ast T + T\ast \nabla^2 T\Big)\right)+ e^{4\phi} P(k+4)\nonumber\\
    &= e^{4\phi} \Delta \nabla^k \Rm + e^{4\phi}\sum_{i=0}^k\sum_{j=0}^i \nabla^{k-i}\phi \ast \nabla^j \Rm \ast \nabla^{i-j}\Rm+ e^{4\phi} \sum_{i=0}^{k-1} \nabla^{k-i}\phi \ast \nabla^{i+2}\Rm \nonumber\\
    &\qquad + \nabla^k \left(e^{4\phi} \Big(\nabla \Rm \ast T + T\ast \nabla^2 T\Big)\right) + e^{4\phi}P(k+4).\nonumber\\
\intertext{Absorbing the second term and part of the third into $P(k+4)$,}
    &= e^{4\phi} \Delta \nabla^k \Rm + e^{4\phi} \nabla \phi \ast \nabla^{k+1} \Rm + \nabla^k \left(e^{4\phi} \Big(\nabla \Rm \ast T + T\ast \nabla^2 T\Big)\right)\nonumber\\
    &\qquad + e^{4\phi}P(k+4).\nonumber\\
\intertext{Applying~\eqref{eqn:ePhiPNotation}, ~\eqref{eqn:nablaDil1}, and absorbing into $P(k+4)$,}
    &= e^{4\phi} \left(\Delta \nabla^k \Rm +  T \ast \nabla^{k+1} \Rm + T \ast \nabla^{k+2}T +  P(k+4)\right).\nonumber
    \end{align} 
We note that the $P(k+4)$ term has no terms including $\nabla^{k+2}\Rm$, $\nabla^{k+1}\Rm$, $\nabla^{k+3}T$, or $\nabla^{k+2}T$. This proves Proposition ~\ref{proposition evolution nabla^k Rm} for $k\geq 0$. 
\end{proof}

\begin{prop}\label{proposition evolution |nabla^k Rm|^2} Let $(\varphi_t, \phi_t)$ be a solution to the generic heterotic $\rG_2$ flow~\eqref{eq: mod_anomaly_flowCbeta} with $C< -\frac{1}{3}$, and with initial condition $(\varphi_0, \phi_0)$ satisfying $d(e^{-4\phi_0} \psi_0)=0$.  Then, for each $k \geq 0$ and each $\epsilon>0$, there is $C(\epsilon, k)< \infty$ such that
\begin{gather*}
\begin{split}
   \frac{\partial}{\partial t} \left|\nabla^k \Rm\right|^2 &\leq e^{4\phi} \Bigg( \Delta \left|\nabla^k \Rm\right|^2 - 2 \left|\nabla^{k+1} \Rm\right|^2 + \epsilon \left(\left|\nabla^{k+1} \Rm\right|^2 + |\nabla^{k+2} T|^2 \right)\\
    &\qquad + C(\epsilon, k) \left(\left|\Rm\right|^{k+3} + |T|^{2k+6} + \sum_{m=1}^{k} \left|\nabla^m \Rm\right|^{\frac{2k+6}{m+2}} + \sum_{m=1}^{k+1} |\nabla^m T|^{\frac{2k+6}{m+1}}\right)\Bigg).
\end{split}
\end{gather*}
If $k=0$, the first sum is zero by definition.
\end{prop}

\begin{proof} 
Applying Proposition \ref{proposition evolution nabla^k Rm} and using that $\frac{\partial g}{\partial t} = e^{4\phi} P(2)$ by ~\eqref{eqn:dgdtschematically},
\begin{align}
       \frac{\partial}{\partial t} &\left|\nabla^k \Rm\right|^2 = \nabla^k \Rm \ast \nabla^k \Rm \ast \frac{\partial g}{\partial t} + 2 \Big\langle \nabla^k \Rm, \frac{\partial}{\partial t} \nabla^k \Rm\Big\rangle\nonumber\\
    &= \nabla^k \Rm \ast \nabla^k \Rm \ast e^{4\phi} P(2) \nonumber\\
    &\quad  + 2 \Big\langle \nabla^k \Rm,  e^{4\phi} \Big(\Delta \nabla^k \Rm+ \nabla^{k+1} \Rm \ast T + T\ast \nabla^{k+2} T + P(k+4)\Big)\Big\rangle\nonumber\\
    &= e^{4\phi}(\nabla^k \Rm)^2 \ast P(2) + e^{4\phi} \left(\Delta \left|\nabla^k \Rm\right| - 2 \left|\nabla^{k+1} \Rm\right|^2 \right) \nonumber\\
    &\qquad + e^{4\phi} \nabla^k \Rm\ast \left(\nabla^{k+1} \Rm \ast T + T \ast \nabla^{k+2} T + P(k+4)\right),\nonumber\\
\intertext{We may absorb the first term into the last using ~\eqref{eqn:MultiplicativePNotation}.}
    &= e^{4\phi} \left(\Delta \left|\nabla^k \Rm\right|^2 - 2 \left|\nabla^{k+1} \Rm\right|^2 + \nabla^k \Rm\ast \nabla^{k+1} \Rm \ast T + \nabla^k \Rm \ast T \ast \nabla^{k+2} T \right) \nonumber\\
    &\qquad \qquad + e^{4\phi} \nabla^k \Rm\ast P(k+4),\nonumber\\
\intertext{By Proposition \ref{proposition evolution nabla^k Rm}, the $P(k+4)$ term has no terms including $\nabla^{k+2}\Rm$, $\nabla^{k+1}\Rm$, $\nabla^{k+3}T$, or $\nabla^{k+2}T$. By ~\eqref{eqn:MultiplicativePNotation}, $\nabla^k \Rm\ast P(k+4) = P(2k + 6)$.}
    &= e^{4\phi} \left(\Delta \left|\nabla^k \Rm\right|^2 - 2 \left|\nabla^{k+1} \Rm\right|^2 + \nabla^k \Rm\ast \nabla^{k+1} \Rm \ast T + \nabla^k \Rm \ast T \ast \nabla^{k+2} T \right) \nonumber\\
    &\qquad \qquad + e^{4\phi} P(2k+6),\nonumber\\
\intertext{The term $P(2k+6)$ does not include any terms with $\nabla^i \Rm$ for $i > k$ or with $\nabla^j T$ for $j > k+1$. Applying Lemma \ref{lemma norm PRmTk} and omitting the appropriate terms from the sum on the right-hand side of~\eqref{eqn:normPRmTk},}
    &\leq e^{4\phi} \left(\Delta \left|\nabla^k \Rm\right|^2 - 2 \left|\nabla^{k+1} \Rm\right|^2 + \nabla^k \Rm \ast \nabla^{k+1} \Rm \ast T + \nabla^k \Rm \ast T \ast \nabla^{k+2} T\right)\nonumber\\
    &\qquad  \qquad + C e^{4\phi} \left(\left|\Rm\right|^{k+3} + |T|^{2k+6} + \sum_{m=1}^{k} \left|\nabla^m \Rm\right|^{\frac{2k+6}{m+2}} + \sum_{m=1}^{k+1} |\nabla^m T|^{\frac{2k+6}{m+1}}\right), \nonumber\\
\intertext{We now apply the Peter--Paul inequality to the terms involving $\nabla^{k+1} \Rm$ and $\nabla^{k+2}T$. So, for each $\epsilon>0$, there is $C(\epsilon, k)< \infty$ such that the following inequality holds:}
    &\leq e^{4\phi} \left(\Delta \left|\nabla^k \Rm\right|^2 - 2 \left|\nabla^{k+1} \Rm\right|^2 + \epsilon \left(\left|\nabla^{k+1} \Rm\right|^2 + |\nabla^{k+2} T|^2 \right) + C(\epsilon, k) \left| \nabla^k \Rm\right|^2 |T|^2\right)\nonumber\\
    &\qquad \qquad + C e^{4\phi} \left(\left|\Rm\right|^{k+3} + |T|^{2k+6} + \sum_{m=1}^{k} \left|\nabla^m \Rm\right|^{\frac{2k+6}{m+2}} + \sum_{m=1}^{k+1} |\nabla^m T|^{\frac{2k+6}{m+1}}\right), \nonumber
    \end{align}
where the first sum is considered to be omitted if $k=0$. We apply Young's inequality to the term with $C(\epsilon, k)$ and absorb it into the second line.
\begin{gather}
    \begin{split}
    &\leq e^{4\phi} \Bigg( \Delta \left|\nabla^k \Rm\right|^2 - 2 \left|\nabla^{k+1} \Rm\right|^2 + \epsilon \left(\left|\nabla^{k+1} \Rm\right|^2 + |\nabla^{k+2} T|^2 \right)\nonumber\\
    &\qquad \qquad + C(\epsilon, k) \left(\left|\Rm\right|^{k+3} + |T|^{2k+6} + \sum_{m=1}^{k} \left|\nabla^m \Rm\right|^{\frac{2k+6}{m+2}} + \sum_{m=1}^{k+1} |\nabla^m T|^{\frac{2k+6}{m+1}}\right)\Bigg).
\end{split}
\end{gather}
This completes the proof of Proposition \ref{proposition evolution |nabla^k Rm|^2}.
\end{proof}

\subsection{Evolution of torsion for the heterotic \texorpdfstring{$\rG_2$}{G2} flow}

For each $k \geq 0$, we compute the evolution of $\nabla^k T$ and estimate the evolution of $\left| \nabla^k T \right|^2$ along heterotic $\rG_2$ flow \eqref{eq:heteroticG2flowintro}, i.e. the generic flow with $C = -\frac{4}{3}$. The reason for this choice of $C$ is that it eliminates a problematic second-order term in the evolution of the torsion.  As in the previous section, our star notation will include implicit terms of the form $\ast \varphi^a \ast \psi^b$.

\begin{prop}\label{proposition evolution nabla^k T} 
Let $(\varphi_t, \phi_t)$ be a solution to the generic heterotic $\rG_2$ flow ~\eqref{eq: mod_anomaly_flowCbeta} with $C = -\frac{4}{3}$ and with initial condition $(\varphi_0, \phi_0)$ satisfying $d(e^{-4\phi_0} \psi_0)=0$. Then, for each $k \geq 0$,
\begin{align}\label{eqn:DerivkT}
     \frac{\partial}{\partial t} \nabla^k T = e^{4\phi} \left(\Delta \nabla^k T +  T \ast (\nabla^{k} \Rm+ \nabla^{k+1}T) +  P(k+3)\right),
\end{align} 
where the $P(k+3)$ in~\eqref{eqn:DerivkT} does not include any terms with $\nabla^{k+1}\Rm$, $\nabla^{k}\Rm$, $\nabla^{k+2}T$, or $\nabla^{k+1}T$.
\end{prop}

\begin{proof}
Since $(\varphi_t, \phi_t)$ is a given solution to the generic heterotic $\rG_2$ flow with conformally coclosed initial condition, we have that $d(e^{-4\phi_t} \psi_t)=0$ for as long as the flow exists. So, the estimates of Section \ref{section: preliminary estimates} hold.

We will begin by studying generic heterotic $\rG_2$ flow and set $C = -\frac{4}{3}$ later.  By Proposition~\ref{p:evolutiongvarphiCbeta}, we may write
\begin{align}
    \frac{\partial}{\partial t}\psi=S\diamond\psi- X\wedge\varphi,
\end{align}
where, with $E$ as above,
\begin{align*}
    S=&\ e^{4\phi}\left(-\Ric_{g} -2\cL_{d\phi}g + E \right), \\
    X=&-\frac{7}{12}e^{4\phi}\left(\left(1+3C\right)d\tau_0 + 9C\tau_0 d\phi\right).
\end{align*}
By Lemma~\ref{lemma evolution T general}, we have
\begin{align}\label{eqn:evolution T anomaly 1}
    \frac{\partial T_{ij}}{\partial t} &= T_{ik}S_{jk}+T_{ik}(X\lrcorner\varphi)_{jk}+\nabla_iX_j-\nabla_kS_{il}\varphi_{jkl}\nonumber\\
    &= T_{ik}S_{jk}+T_{ik}(X\lrcorner\varphi)_{jk} -\frac{7}{12}e^{4\phi}\left(\left(1+3C\right)\nabla_i \nabla_j \tau_0 + 9C\nabla_i(\tau_0 \nabla_j\phi)\right) - 4 \nabla_k \phi \,S_{il} \varphi_{jkl} \nonumber\\
    &\quad + e^{4\phi} \nabla_k\left(R_{il}+4\nabla_i \nabla_l \phi + E \right) \varphi_{jkl}.
\end{align}
By Lemma \ref{lm: Ricci and scalar curv}, 
\begin{align}\label{eqn:curlofRic}
    \nabla_k R_{il} \varphi_{jkl} &= \nabla_k\left(\left(\nabla_iT_{mn}-\nabla_mT_{in}\right)\varphi_{lmn}+\tr(T)T_{il}-T_{im}T_{ml}+T_{im}T_{np}\psi_{mnpl}\right) \varphi_{jkl}\nonumber\\
    &= \nabla_k \left(\nabla_i T_{mn} \varphi_{lmn}\right)\varphi_{jkl} - \nabla_k \left(\nabla_m T_{in}\varphi_{lmn}\right)\varphi_{jkl}\nonumber\\
    &\quad + \nabla_k\left(\tr(T)T_{il}-T_{im}T_{ml}+T_{im}T_{np}\psi_{mnpl}\right)\varphi_{jkl}.
\end{align}
For a conformally coclosed $\rG_2$-structure, $T = \tau_0 g + T_{\sym} - d \phi \lrcorner \varphi$. Using this and ~\eqref{eq: varphi1 varphi1}, we analyze the term
\begin{align}\label{equation ev T 1}
    \nabla_k\left(\nabla_iT_{mn}\varphi_{lmn}\right)\varphi_{jkl}=&\ \nabla_k\left(\nabla_i(T_{mn}\varphi_{mnl})-T_{mn}T_{ip}\psi_{pmnl}\right)\varphi_{jkl}\nonumber\\
    &\ -6\nabla_k\nabla_i\nabla_l\phi\varphi_{jkl}-\nabla_k\left(T_{mn}T_{ip}\psi_{pmnl}\right)\varphi_{jkl}\nonumber\\    =&\ -6\nabla_i\nabla_k\nabla_l\phi\varphi_{jkl}-6R_{kilp}\nabla_p\phi\varphi_{jkl}-\nabla_k\left(T_{mn}T_{ip}\psi_{pmnl}\right)\varphi_{jkl}\nonumber\\    
    =&\  -6R_{kilp}\nabla_p\phi\varphi_{jkl}-\nabla_k\left(T_{mn}T_{ip}\psi_{pmnl}\right)\varphi_{jkl}.
\end{align}
Next, we analyze the term
\begin{align}\label{equation ev T 2}
    -\nabla_k & \left(\nabla_m T_{in}\varphi_{lmn}\right)\varphi_{jkl} = -\nabla_k\nabla_m T_{in} \,\varphi_{lmn}\varphi_{jkl} - \nabla_m T_{in} \nabla_k \varphi_{lmn}\,\varphi_{jkl} \nonumber
\intertext{Applying ~\eqref{eq: varphi1 varphi1} and then commuting covariant derivatives,}
    =&\ - \nabla_k \nabla_j T_{ik} + \Delta T_{ij} - \nabla_k \nabla_m T_{in} \psi_{mnjk}- \nabla_m T_{in} \nabla_k \varphi_{lmn}\,\varphi_{jkl}\nonumber\\
    =&\ -\nabla_j \nabla_k T_{ik} - R_{kjip}T_{pn} - R_{kjkp} T_{ip}+ \Delta T_{ij}\nonumber\\
    &\ \quad - \frac{1}{2}(\nabla_k \nabla_m T_{in} - \nabla_m \nabla_k T_{in}) \psi_{mnjk} - \nabla_m T_{in} \nabla_k \varphi_{lmn}\,\varphi_{jkl}\nonumber\\
    =&\ -\nabla_j \nabla_k T_{ik} - R_{kjip}T_{pn} - R_{kjkp} T_{ip}+ \Delta T_{ij}\nonumber\\
    &\quad - \frac{1}{2}\left(R_{kmip}T_{pn} + R_{kmnp}T_{ip}\right) \psi_{mnjk} - \nabla_m T_{in} \nabla_k \varphi_{lmn}\,\varphi_{jkl},\nonumber\\
\intertext{Applying the $\rG_2$-Bianchi identity~\eqref{eq: G2Bianchi_ident},}
    =&\ -\nabla_j \nabla_i (\tr T)  - \nabla_j \left(\frac{1}{2}R_{kiab} \varphi_{kab} + T_{ka}T_{ib}\varphi_{kab}\right) - R_{kjip}T_{pn} - R_{kjkp} T_{ip}\nonumber\\
    &\quad + \Delta T_{ij} - \frac{1}{2}\left(R_{kmip}T_{pn} + R_{kmnp}T_{ip}\right) \psi_{mnjk} - \nabla_m T_{in} \nabla_k \varphi_{lmn}\,\varphi_{jkl},\nonumber\\
    =&\ - \frac{7}{4} \nabla_i \nabla_j \tau_0 - \nabla_j \left(T_{ka}T_{ib}\varphi_{kab}\right) - R_{kjip}T_{pn} - R_{kjkp} T_{ip}\nonumber\\
    &\quad + \Delta T_{ij} - \frac{1}{2}R_{kmip}T_{pn} \psi_{mnjk} - \nabla_m T_{in} \nabla_k \varphi_{lmn}\,\varphi_{jkl},
    \end{align}
where the last line follows from ~\eqref{eq:trT} and the fact that $R_{kiab}\varphi_{kab} = R_{kmnp}\psi_{mnjk} = 0$, which follows from the first Bianchi identity.

Applying~\eqref{equation ev T 1} and ~\eqref{equation ev T 2} to~\eqref{eqn:curlofRic},
\begin{align}\label{equation curl of Ric 2}
        \nabla_k R_{il} \varphi_{jkl} &= -6R_{kilp}\nabla_p\phi\varphi_{jkl}-\nabla_k\left(T_{mn}T_{ip}\psi_{pmnl}\right)\varphi_{jkl} \nonumber\\        &\quad - \frac{7}{4} \nabla_i \nabla_j \tau_0 - \nabla_j \left(T_{ka}T_{ib}\varphi_{kab}\right) - R_{kjip}T_{pn} - R_{kjkp} T_{ip}\nonumber\\
    &\quad + \Delta T_{ij} - \frac{1}{2}R_{kmip}T_{pn} \psi_{mnjk} - \nabla_m T_{in} \nabla_k \varphi_{lmn}\,\varphi_{jkl}\nonumber\\
    &\quad + \nabla_k\left(\tr(T)T_{il}-T_{im}T_{ml}+T_{im}T_{np}\psi_{mnpl}\right)\varphi_{jkl}.
\end{align}
Next, we analyze the term, from ~\eqref{eqn:evolution T anomaly 1},
\begin{align}\label{equation ev T 3}
    \nabla_k \nabla_i \nabla_l \phi \,\varphi_{jkl} &= \nabla_i \nabla_k \nabla_l \phi \,\varphi_{jkl} + R_{kipl}\nabla_p \phi\, \varphi_{jkl}
    = R_{kipl}\nabla_p \phi\, \varphi_{jkl}.
\end{align}
Applying~\eqref{equation curl of Ric 2} and ~\eqref{equation ev T 3} to ~\eqref{eqn:evolution T anomaly 1},
\begin{align}\label{eqn:evolution T anomaly 2}
    e^{-4\phi}\frac{\partial T_{ij}}{\partial t} &= \Delta T_{ij}  - \frac{7}{4} \left(C + \frac{4}{3}\right) \nabla_i \nabla_j \tau_0 + e^{-4\phi}T_{ik}S_{jk}+e^{-4\phi}T_{ik}(X\lrcorner\varphi)_{jk} -\frac{7}{12}\left(9C\nabla_i(\tau_0 \nabla_j\phi)\right) \nonumber\\
    &\quad - 4 e^{-4\phi}\nabla_k \phi \,S_{il} \varphi_{jkl} +  \nabla_k E_{il} \varphi_{jkl}\nonumber\\
    &\quad + 4R_{kipl}\nabla_p \phi\, \varphi_{jkl}  -6R_{kilp}\nabla_p\phi\varphi_{jkl}-\nabla_k\left(T_{mn}T_{ip}\psi_{pmnl}\right)\varphi_{jkl} - \nabla_j \left(T_{ka}T_{ib}\varphi_{kab}\right) \nonumber\\
    &\quad  - R_{kjip}T_{pn} - R_{kjkp} T_{ip} - \frac{1}{2}R_{kmip}T_{pn} \psi_{mnjk} - \nabla_m T_{in} \nabla_k \varphi_{lmn}\,\varphi_{jkl}\nonumber\\
    &\quad + \nabla_k\left(\tr(T)T_{il}-T_{im}T_{ml}+T_{im}T_{np}\psi_{mnpl}\right)\varphi_{jkl}.
\end{align}

Using ~\eqref{eqn:DerivativesDilaton} -- \eqref{eqn:Derivativesphipsi}, the expressions for $S$ and $X$, and the fact that $E = T^2 \ast (1+ \varphi + \psi)$, we have that
\begin{align}\label{eqn: flow of T with arbitrary C}
    \frac{\partial T_{ij}}{\partial t} = e^{4\phi} \left( \Delta T_{ij} - \frac{7}{4} \left( C+ \frac{4}{3}\right) \nabla_i \nabla_j \tau_0 + T \ast (\Rm + \nabla T + T^2)\right),
\end{align}
where the star notation as always dependence on $C, \gamma, \sigma$ as well as extra terms of the form $\ast \varphi^a \ast \psi^b$.  Setting $C = - \frac{4}{3}$, this proves Proposition \ref{proposition evolution nabla^k T} for $k=0$.

We now compute the evolution of $\nabla^k T$ schematically for $k>0$. 
Applying the commutator of the time derivative with $\nabla^k$, using ~\eqref{eqn: flow of T with arbitrary C} and that $\frac{\partial g}{\partial t} = e^{4\phi} P(2)$ as in ~\eqref{eqn:dgdtschematically},
\begin{align}
    \frac{\partial}{\partial t} &\nabla^k T = \nabla^k \left(\frac{\partial}{\partial t} T\right) + \sum_{i=1}^k \nabla^{k-i} T \ast \nabla^i \frac{\partial g}{\partial t}\nonumber\\
    &= \nabla^k \left(e^{4\phi}\left(\Delta T + T \ast (\nabla T +\Rm + T^2) \right)\right) + \sum_{i=1}^k \nabla^{k-i} T \ast \nabla^i \left(e^{4\phi} P(2)\right),\nonumber\\
\intertext{Applying~\eqref{eqn:DerivativePNotation}, ~\eqref{eqn:MultiplicativePNotation}, and ~\eqref{eqn:ePhiPNotation},}
    &= \nabla^k \left(e^{4\phi}\left(\Delta T + T \ast (\nabla T +\Rm + T^2) \right)\right) +T \ast \nabla^k \left(e^{4\phi}P(2) \right)\nonumber\\
    &\qquad + e^{4\phi} P(k+3),\nonumber\\
\intertext{where the $P(k+3)$ includes no terms with $\nabla^{k+1}\Rm$, $\nabla^{k}\Rm$, $\nabla^{k+2}T$, or $\nabla^{k+1}T$. Expanding the second term and absorbing part into the third,}
    &= \nabla^k \left(e^{4\phi}\left(\Delta T + T \ast (\nabla T +\Rm + T^2)\right)\right) + e^{4\phi}T \ast (\nabla^k \Rm + \nabla^{k+1}T)\nonumber\\
    &\qquad + e^{4\phi} P(k+3),\nonumber\\
\intertext{Expanding the first term and then applying the commutator of $\nabla^k$ and $\Delta$ (see~\cite[Section 7.2]{KnopfBook}),}
    &= e^{4\phi} \sum_{i=0}^k \nabla^{k-i} \phi \ast \nabla^i \Delta T + \nabla^k \left(e^{4\phi}T \ast (\nabla T +\Rm + T^2)\right)+ e^{4\phi}T \ast (\nabla^k \Rm + \nabla^{k+1}T) \nonumber\\
    &\qquad + e^{4\phi} P(k+3)\nonumber\\
    &= e^{4\phi} \Delta \nabla^k T + e^{4\phi}\sum_{i=0}^k\sum_{j=0}^i \nabla^{k-i}\phi \ast \nabla^j \Rm \ast \nabla^{i-j}T + e^{4\phi} \sum_{i=0}^{k-1} \nabla^{k-i}\phi \ast \nabla^{i+2}T \nonumber\\
    &\qquad + \nabla^k \left(e^{4\phi}T \ast (\nabla T +\Rm + T^2) \right)+ e^{4\phi}T \ast (\nabla^k \Rm + \nabla^{k+1}T) + e^{4\phi}P(k+3),\nonumber\\
\intertext{Absorbing the second and third terms into the last two terms and using~\eqref{eqn:nablaDil1}, which brings a term of the form $(1+\varphi^a + \psi^b)$ which can be absorbed into the star notation,}
    &= e^{4\phi} \Delta \nabla^k T + \nabla^k \left(e^{4\phi}T \ast (\nabla T +\Rm + T^2)\right)  + e^{4\phi} T \ast (\nabla^k \Rm + \nabla^{k+1}T) \nonumber\\
    &\qquad + e^{4\phi}P(k+3),\nonumber\\
\intertext{Expanding the second term, applying~\eqref{eqn:MultiplicativePNotation} and~\eqref{eqn:ePhiPNotation}, and absorbing into $P(k+3)$,}
    &= e^{4\phi} \left(\Delta \nabla^k T +  T \ast (\nabla^{k} \Rm+ \nabla^{k+1}T) +  P(k+3)\right).\nonumber
    \end{align} 
By analyzing the source of the $P(k+3)$ term in~\eqref{eqn:DerivkT}, we find that it contains no terms with $\nabla^{k+1}\Rm$, $\nabla^{k}\Rm$, $\nabla^{k+2}T$, or $\nabla^{k+1}T$. This proves Proposition \ref{proposition evolution nabla^k T} for $k\geq 0$. 
\end{proof}

\begin{prop}\label{proposition evolution |nabla^k T|^2} Let $(\varphi_t, \phi_t)$ be a solution to the generic heterotic $\rG_2$ flow ~\eqref{eq: mod_anomaly_flowCbeta} with $C = -\frac{4}{3}$ and with initial condition $(\varphi_0, \phi_0)$ satisfying $d(e^{-4\phi_0} \psi_0)=0$.  Then, for each $k \geq 0$ and each $\epsilon>0$, there is $C(\epsilon, k, \gamma, \sigma)< \infty$ such that
\begin{align}
   \frac{\partial}{\partial t} \left|\nabla^k T\right|^2 &\leq e^{4\phi} \Big( \Delta |\nabla^k T|^2 - 2 |\nabla^{k+1} T|^2 + \epsilon \left(|\nabla^{k} \Rm|^2 + |\nabla^{k+1} T|^2 \right)\nonumber\\
    &\qquad + C(\epsilon, k, \gamma,\sigma) \Big(|\Rm|^{k+2} + |T|^{2k+4} + \sum_{m=1}^{k-1} |\nabla^m \Rm|^{\frac{2k+4}{m+2}} + \sum_{m=1}^{k} |\nabla^m T|^{\frac{2k+4}{m+1}}\Big)\Big).
\end{align}
By definition, both sums are zero if $k=0$, and the first sum is zero if $k=1$.
\end{prop}
\begin{proof}
We compute the evolution of $|\nabla^k T|^2$, using Proposition \ref{proposition evolution nabla^k T} and that $\frac{\partial g}{\partial t} = e^{4 \phi} P(2)$.
\begin{align}
       \frac{\partial}{\partial t} &|\nabla^k T|^2 = \nabla^k T \ast \nabla^k T \ast \frac{\partial g}{\partial t} + 2 \Big\langle \nabla^k T, \frac{\partial}{\partial t} \nabla^k T\Big\rangle\nonumber\\
    &= \nabla^k T \ast \nabla^k T \ast e^{4\phi} P(2) \nonumber\\
    &\quad  + 2 \Big\langle \nabla^k T,  e^{4\phi} \left(\Delta \nabla^k T +  T \ast (\nabla^{k} \Rm+ \nabla^{k+1}T) +  P(k+3)\right)\Big\rangle\nonumber\\
    &= e^{4\phi}(\nabla^k T)^2 \ast P(2) + e^{4\phi} \left(\Delta |\nabla^k T| - 2 |\nabla^{k+1} T|^2 \right) \nonumber\\
    &\qquad + e^{4\phi} \nabla^k T \ast \left(T \ast (\nabla^{k} \Rm+ \nabla^{k+1}T) +  P(k+3)\right),\nonumber\\
\intertext{We may absorb the first term into the last using ~\eqref{eqn:MultiplicativePNotation}.}
    &= e^{4\phi} \left(\Delta |\nabla^k T|^2 - 2 |\nabla^{k+1} T|^2 \right)  + e^{4\phi} \nabla^k T \ast \left(T \ast (\nabla^{k} \Rm+ \nabla^{k+1}T) +  P(k+3)\right),\nonumber\\
\intertext{By Proposition \ref{proposition evolution nabla^k T}, $P(k+3)$ has no terms including $\nabla^{k+1}\Rm$, $\nabla^{k}\Rm$, $\nabla^{k+2}T$, or $\nabla^{k+1}T$. 
By ~\eqref{eqn:MultiplicativePNotation}, $\nabla^k T \ast P(k+3) = P(2k + 4)$.}
    &= e^{4\phi} \left(\Delta |\nabla^k T|^2 - 2 |\nabla^{k+1} T|^2 \right)  + e^{4\phi} \nabla^k T \ast T \ast (\nabla^{k} \Rm+ \nabla^{k+1}T) + e^{4\phi} P(2k+4),\nonumber\\
\intertext{The term $P(2k+4)$ does not include any terms with $\nabla^i \Rm$ for $i > k-1$ or with $\nabla^j T$ for $j > k$. Applying Lemma \ref{lemma norm PRmTk} and omitting the appropriate terms from the sum on the right-hand side of~\eqref{eqn:normPRmTk},}    
&\leq e^{4\phi} \left(\Delta |\nabla^k T|^2 - 2 |\nabla^{k+1} T|^2 +\nabla^k T \ast T \ast (\nabla^{k} \Rm+ \nabla^{k+1}T) \right)\nonumber\\
    &\qquad + C e^{4\phi} \left(\left|\Rm\right|^{k+2} + |T|^{2k+4} + \sum_{m=1}^{k-1} \left|\nabla^m \Rm\right|^{\frac{2k+4}{m+2}} + \sum_{m=1}^{k} |\nabla^m T|^{\frac{2k+4}{m+1}}\right), \nonumber\\
\intertext{and the Peter--Paul inequality implies that for each $\epsilon>0$, there is $C(\epsilon,k, \gamma,\sigma)< \infty$ such that the following inequality holds:}
    &\leq e^{4\phi} \Big( \Delta |\nabla^k T|^2 - 2 |\nabla^{k+1} T|^2 + \epsilon \left(|\nabla^{k} \Rm|^2 + |\nabla^{k+1} T|^2 \right)\nonumber\\
    &\qquad \qquad + C(\epsilon, k, \gamma,\sigma) \Big(\left|\Rm\right|^{k+2} + |T|^{2k+4} + \sum_{m=1}^{k-1} \left|\nabla^m \Rm\right|^{\frac{2k+4}{m+2}} + \sum_{m=1}^{k} \left|\nabla^m T\right|^{\frac{2k+4}{m+1}}\Big)\Big).
\end{align}
\end{proof}

\subsection{Shi-type estimates for the heterotic \texorpdfstring{$\rG_2$}{G2} flow}

We prove Shi-type smoothing estimates for the heterotic $\rG_2$ flow closely following G.\ Chen's proof of \cite[Theorem 2.1]{GaoChenShi}.
\begin{thm} \label{t:maththm2bulk}
    Let $(\varphi_t, \phi_t)$ be a solution to the generic heterotic $\rG_2$ flow ~\eqref{eq: mod_anomaly_flowCbeta} with $C = -\frac{4}{3}$ and with initial condition $(\varphi_0, \phi_0)$ satisfying $d(e^{-4\phi_0} \psi_0)=0$. Let $B_r(p)$ be a ball of radius $r$ around $p \in M$ with respect to the metric $g_{\varphi_0}$. Suppose that there exist $t_0, \Lambda>0$ such that
\begin{align}\label{equation a priori bound shi estimates}
        \left|\Rm\right| + |T|^2 + |\nabla T| + |\phi| < \Lambda
\end{align}
on $B_r(p) \times [0, t_0]$.  Then, for each $k \geq 0$, there is a constant $C(k, t_0, \Lambda, \gamma,\sigma, r)$ such that
\begin{align}\label{equation theorem shi result}
    \left|\nabla^{k+2}\phi\right| + \left| \nabla^k \Rm\right| + \left|\nabla^{k+1} T\right| < C(k, t_0, \Lambda, \gamma,\sigma, r)
\end{align}
on $B_{r/2}(p) \times [t_0/2, t_0]$.
\end{thm}
\begin{proof}
First, using Proposition \ref{proposition evolution nabla^k T} with $k=0$, we may compute that
\begin{align}\label{equation evolution |T|^4}
    \frac{\partial}{\partial t}|T|^4 \leq e^{4\phi} \left( \Delta |T|^4 + C(|T|^6 + \left|\Rm\right|^3 + \left|\nabla T\right|^3)\right),
\end{align}
where $C< \infty$.

Combining Proposition~\ref{proposition evolution |nabla^k Rm|^2}, Proposition~\ref{proposition evolution |nabla^k T|^2}, and ~\eqref{equation evolution |T|^4}, we have that there is some $C < \infty$ such that
\begin{align}\label{equation shi 1}
    \left(\frac{\partial}{\partial t} - \Delta\right)\left(\left|\Rm\right|^2 + |T|^4 + \left|\nabla T\right|^2 + 1\right) &\leq   - e^{4\phi}(|\nabla \Rm|^2 + |\nabla^2 T|^2)\nonumber\\
    &\qquad + C e^{4\phi}(|\Rm|^3 + |\nabla T|^3 + |T|^6 + 1).
\end{align}
Likewise, for each $k \geq 1$, there is $C = C(k, \gamma,\sigma) < \infty$ such that
\begin{align}\label{equation shi 2}
    \left(\frac{\partial}{\partial t} - \Delta\right)&\left(\left|\nabla^k\Rm\right|^2 + \left|\nabla^{k+1} T\right|^2\right) \leq   - e^{4\phi}(|\nabla^{k+1} \Rm|^2 + |\nabla^{k+2} T|^2)\nonumber\\
    &\qquad \qquad \qquad \quad + C e^{4\phi}\left(\sum_{m=0}^k \left( \left|\nabla^m \Rm\right|^{\frac{2k+6}{m+2}} + |\nabla^{m+1} T|^{\frac{2k+6}{m+2}}\right)+ |T|^{2k+6} + 1\right).
\end{align}

These evolution equations are exactly of the type that allow one to find smoothing estimates via Shi's technique. G.\ Chen has provided a framework for Shi-type estimates for general flows of $\rG_2$-structures in ~\cite{GaoChenShi}. The flow we are considering essentially fits into Chen's framework, but the only distinction is that we have an additional flow of a dilaton. We give some details to explain how our dilaton flow fits with this framework.

Let $\mu \in \mathbb{R}$ be a constant to be determined later in terms of $\Lambda$. As in the proof of~\cite[Theorem 2.1]{GaoChenShi}, define
\begin{align}
    Q := \left(\left|\Rm\right|^2 + |T|^4 + \left|\nabla T\right|^2 + \mu\right) \left(\left|\nabla \Rm\right|^2 + \left|\nabla^{2} T\right|^2\right).
\end{align}
Then, using ~\eqref{equation shi 1} and ~\eqref{equation shi 2} as in ~\cite{GaoChenShi}, we find that for $\mu$ chosen depending on $\Lambda$,
\begin{align}
    \left(\frac{\partial}{\partial t}  - \Delta \right) Q &\leq e^{4\phi} \left( - C(\Lambda, \gamma,\sigma) Q^2 + C(\Lambda, \gamma,\sigma)\right) \nonumber\\\label{equation PDI for Q}
    &\leq - C_1(\Lambda, \gamma,\sigma) Q^2 + C_2(\Lambda, \gamma,\sigma),
\end{align}
where in the second line we used that $|\phi| \leq \Lambda$. 

Now, due to the bound ~\eqref{equation a priori bound shi estimates} and the evolution equations Proposition \ref{p:evolutiongvarphiCbeta} and \eqref{eq: ddt Gamma}, we have that 
\begin{align}\label{equation: dg/dt shi}
    \left| \frac{\partial g_{ij}}{\partial t}\right| &\leq C(\Lambda, \gamma,\sigma),\\\label{equation: dGamma/dt shi}
    \left| \frac{\partial }{\partial t}\Gamma_{ij}^k\right| &\leq C(\Lambda, \gamma,\sigma, t_0) \left(|\nabla \Rm| + |\nabla^2 T|^2 + 1\right),
\end{align}
where we used our expressions for $H^2_{\varphi}$ and our estimates on derivatives of the dilaton from Section \ref{section: preliminary estimates}.

For some $\nu>0$ and a cut-off function $\chi: M \to \mathbb{R}$ such that $\chi \equiv 1$ on $B_{r/2}(p)$ and $\chi \equiv 0$ on $M \setminus B_r(p)$, we define
\begin{align}
    W = \frac{\nu}{\chi^2} + \frac{1}{C_1(\Lambda)t} + \sqrt{\frac{C_2(\Lambda)}{C_1(\Lambda)}}.
\end{align}
Exactly as in the proof of~\cite[Theorem 2.1]{GaoChenShi}, we can use   ~\eqref{equation: dg/dt shi} and ~\eqref{equation: dGamma/dt shi} to find $\nu>0$ such that 
\begin{align}\label{equation PDI for W}
    \left(\frac{\partial}{\partial t}  - \Delta \right) W > - C_1(\Lambda, \gamma,\sigma) W^2 + C_2(\Lambda, \gamma,\sigma)
\end{align}
at the first time $t^*$ when $\sup_M (Q- W)=0$ and at any point $q \in M$ where this supremum is realized.

Combining ~\eqref{equation PDI for Q} with ~\eqref{equation PDI for W}, we have that at $(q,t^*)$,
\begin{align}
    \left(\frac{\partial}{\partial t}  - \Delta \right) (Q-W) < 0. 
\end{align}
However, since $Q<W$ for $t \in [0, t^*)$, we have that at $(q, t^*)$,
\begin{align}
    \left(\frac{\partial}{\partial t}  - \Delta \right) (Q-W) \geq 0.
\end{align}
This is a contradiction, so we find that $t^* = t_0$ and so $Q \leq W$ for $t \in [0, t_0]$. This proves that
\begin{align}\label{equation shi estimate k=1}
     \left|\nabla \Rm\right| + \left|\nabla^{2} T\right| < C(1, t_0, \Lambda, \gamma,\sigma, r).
\end{align}
on $B_r(p) \times [t_0/2, t_0]$. By ~\eqref{eqn:DerivativesDilaton}, $\nabla^3 \phi = P(3)$. So, ~\eqref{equation shi estimate k=1} combined with Lemma \ref{lemma norm PRmTk} implies that ~\eqref{equation theorem shi result} holds. This proves the theorem for $k=1$. 

For $k >1$, we follow a similar logic as in~\cite{GaoChenShi}, using the quantity
\begin{align}
    Q_k = \left(|\nabla^k \Rm|^2 + |\nabla^{k+1} T|^2 + \mu_k\right) \left(|\nabla^{k+1}\Rm|^2+ |\nabla^{k+2} T|^2\right),
\end{align}
where $\mu_k \in \mathbb{R}$ is chosen depending on $\Lambda$ in order for $Q_k$ to satisfy a partial differential inequality as in ~\eqref{equation PDI for Q}. We use a similar $W$ and induct on $k$, using the bounds for $1, \dots, k-1$. Then, we find that
\begin{align}\label{equation shi estimate k}
     \left|\nabla^k \Rm\right| + \left|\nabla^{k+1} T\right| < C(k, t_0, \Lambda, \gamma,\sigma, r).
\end{align}
By~\eqref{eqn:DerivativesDilaton} and Lemma~\ref{lemma norm PRmTk} again, we find that~\eqref{equation theorem shi result} holds, proving the theorem for $k > 1$.
\end{proof}

\subsection{Convergence of nonsingular solutions}

With the analytic theory developed so far, we obtain a natural corollary on the convergence of nonsingular solutions.  The proof employs many standard techniques in the theory of geometric flows so we only provide a brief sketch.

\begin{cor} \label{c:nonsingularbulk} 
    Let $(\varphi_t, \phi_t)$ be a solution to the heterotic $\rG_2$ flow ~\eqref{eq:heteroticG2flowintro} on a compact manifold $M$ with 
    $\gamma$ and $\sigma$ satisfying the hypotheses of Theorem \ref{thm: monotonicity funct}, and with initial condition $(\varphi_0, \phi_0)$ satisfying $d(e^{-4\phi_0} \psi_0)=0$.  Furthermore assume there exists a constant $\gL > 0$ such that for all $t > 0$ for which the flow exists,
    \begin{align*}
        \max \{ \inj^{-1}_{g}, \diam(g), \left|\Rm\right|, |T|^2, |\nabla T|, |\phi| \} < \Lambda.
    \end{align*}
    Then the flow exists on $[0,\infty)$, and for any sequence $\{t_k\} \to \infty$ there exists a subsequence such that $(\varphi_{t_{k_j}}, \phi_{t_{k_j}})$ converges to $(\varphi_{\infty}, \phi_{\infty})$ where $\varphi_{\infty}$ is a torsion-free $\rG_2$-structure and $\phi_{\infty}$ is constant.
    \begin{proof} By Theorem \ref{t:maththm2bulk} we obtain uniform a priori estimates on all derivatives of curvature and torsion for any time $t_0 > 0$.  It follows by standard compactness arguments developed in the theory of Ricci flow (cf. \cite{HamCompactness}), together with the short-time existence result of Theorem \ref{thm:STE}, that the flow exists on $[0,\infty)$.  Using the assumed geometric bounds it also follows that the $\rG_2$-dilaton functional $\mathcal M$ is bounded above along the flow.  Using the monotonicity of $\mathcal M$ from Theorem \ref{thm: monotonicity funct} together with the assumed geometric bounds and again employing results from compactness theory, the claim of subsequential convergence follows.
    \end{proof}
\end{cor}

\begin{rmk} 
One expects genuine convergence of flow lines, which would follow from a dynamic stability result for heterotic $\rG_2$ flow. Our method of proof shall be compared with that of \cite[Theorem 2]{PPZ3}, for the \emph{anomaly flow} of $\operatorname{SU}(3)$-structures.
\end{rmk}

\section{Dimensional reduction of the generic heterotic \texorpdfstring{$\rG_2$}{G2} flow}\label{sec:dimred}
\subsection{\texorpdfstring{$\mathbb{S}^1$}{S1}-invariant \texorpdfstring{$\rG_2$}{G2}-structures} 
Let $N$ be a $6$-manifold with $\SU(3)$-structure $(\omega,\rho_+)\in \Omega^2(N)\times\Omega^3(N)$, i.e.\
$$
  \frac{\omega^3}{3!}=\frac14\rho_+\wedge\rho_-=\vol_6,
$$
here, the metric $g_6$ and the almost complex structure $J$ on $N$, satisfy $\omega=g_6(J\cdot,\cdot)$ and $\rho_-=J\rho_+=\star_6\rho_+$. The space of differential $k$-forms on $N$ has the decomposition
\begin{align}\label{eq: su3 decomposition forms}
    \Omega^2(N)=\Omega^2_1\oplus\Omega^2_6\oplus\Omega^2_8, \qquad \Omega^3=\Omega^3_{1+}\oplus\Omega^3_{1-}\oplus\Omega^3_6\oplus\Omega^3_{12}
\end{align}
where
\begin{align} \nonumber
    \Omega^2_1=&\ \{f\omega : \quad f\in C^{\infty}(N,\bR)\}\\ \nonumber
    \Omega^2_6=&\ \{\beta\in\Omega^2(N) : \quad J\beta=-\beta\}=\{\beta\in\Omega^2(N) :\quad \star_6(\beta\wedge\omega)=\beta\}\\ \nonumber
    =&\ \{\star_6(\alpha\wedge\rho_-)  : \quad \alpha\in \Omega^1(N)\}\\ \label{eq: su3 components forms}
    \Omega^2_8=&\ \{\beta\in\Omega^2(N) : \quad J\beta=\beta, \quad \beta\wedge\omega^2=0\}=\{\beta\in\Omega^2(N)  :\quad \star_6(\beta\wedge\omega)=-\beta\}\\ \nonumber
 \Omega^3_{1\pm}=&\ \{f\rho_{\pm}: \quad f\in C^{\infty}(N,\bR)\}&\\ \nonumber
 \Omega^3_6=&\{\alpha\wedge\omega  : \quad \alpha\in \Omega^1(N)\}\\ \nonumber
 \Omega^3_{12}=&\ \{\beta\in \Omega^3_{12} : \quad \beta\wedge\rho_{\pm}=0, \quad \beta\wedge\omega=0\}.
\end{align}
The torsion forms of an $\SU(3)$-structure are \cite{chiossi2002, bedulli2007ricci}
\begin{align}\label{eq: torsion_SU3}\nonumber
    d\omega=&\ \sigma_0\rho_++\pi_0\rho_-+\nu_1\wedge\omega+\nu_3,\\
    d\rho_+=&\ \frac{2}{3}\pi_0\omega^2+\pi_1\wedge\rho_+-\pi_2\wedge\omega,\\ \nonumber
     d\rho_-=&-\frac{2}{3}\sigma_0\omega^2+\pi_1\wedge\rho_--\sigma_2\wedge\omega,
\end{align}
where $\pi_0,\sigma_0\in C^\infty(N)$, $\pi_1,\nu_1\in\Omega^1(N)$, $\pi_2,\sigma_2\in \Omega^2_8$ and $\nu_3\in\Omega^3_{12}$. Let $M^7$ be a principal $\bS^1$-bundle over $N^6$ and $\theta$ a $1$-form connection on $M$. Consider the $\rG_2$-structure on $M$ 
\begin{equation}\label{eq: S1-inv. G2-struct}\varphi=h^{-3/4}\omega\wedge\theta+\rho_+, \qwhereq h\in C^\infty(N,\bR_+).
\end{equation}
 Thus, the metric, the volume form, the Hodge star and the $4$-form $\psi$ are given by (see \cite[(3.1)]{foscolo2021complete})
\begin{align}\label{eq: metric+vol}
g_7=&\ g_6+h^{-3/2}\theta^2, \quad \vol_7=h^{-3/4}\vol_6\wedge \theta\\ \label{eq: star_7}
    \star_7\beta=&\ h^{-3/4}\star_6\beta\wedge \theta, \quad \star_7(\beta\wedge \theta)=(-1)^kh^{3/4}\star_6\beta, \qforq \beta\in \Omega^k(N)\\ \label{eq: psi_S1-inv}
     \psi=&\ \star_7\varphi= \frac{\omega^2}{2}+h^{-3/4}\rho_-\wedge \theta.
\end{align}
We recall the following $\SU(3)$-identities
\begin{align}\label{eq: rho^omega}
       \star_6(\alpha\wedge\rho_-)\wedge\omega=&\ J\alpha\wedge\rho_+=\alpha\wedge\rho_-, \\  \label{eq: rho^rho}\star_6(\alpha\wedge\rho_-)\wedge\rho_+=&\ \alpha\wedge\omega^2=2\star_6J\alpha,\\ \label{eq: rho_ and rho -}
       \star_6(\alpha\wedge\rho_-)\wedge\rho_-=&\ -J\alpha\wedge\omega^2=2\star_6\alpha.
   \end{align}
 The next proposition describes integrable $\rG_2$-structures, our result is equivalent to \cite[Theorem 5.5]{FinoFowdar} up to conformal change, also,  we give here an equivalent presentation in terms of the torsion forms \eqref{eq: torsion_SU3}:

\begin{prop}\label{prop: tau2=0 SU3}
    Let $(M^7,\theta)$ and $(N^6,\omega,\rho_+)$ as before. The $\rG_2$-structure \eqref{eq: S1-inv. G2-struct} is integrable if and only if,
    \begin{equation}\label{eq: integral_condition}
        -h^{3/4}d h^{-3/4}+2\nu_1-\pi_1=h^{-3/4}\star_6(F_\theta\wedge\rho_+) \qandq \sigma_2=0.
    \end{equation}
    In particular, the torsion forms $\tau_0$ and $\tau_1$ are:
    \begin{align}\label{eq: tau_0 SU3}
        \tau_0=\frac{2}{7}\left(\frac{1}{h^{3/4}}(\omega\lrcorner F_\theta)+4\pi_0\right), \quad \tau_1= \frac12\nu_1-\frac{1}{4h^{3/4}}\star_6(F_\theta\wedge\rho_+)-\frac{\sigma_0}{3h^{3/4}}\theta.
    \end{align}
    where $F_\theta=d\theta$ is the curvature $2$-form.
\end{prop}

\begin{proof}
Using \eqref{eq: torsion_SU3}, the exterior derivative of $\psi$ is
\begin{align*}
    d\psi=&\left(dh^{-3/4}\wedge\rho_-+\frac{2}{3h^{3/4}}\sigma_0\omega^2+h^{-3/4}\pi_1\wedge\rho_--\sigma_2\wedge\omega\right)\wedge\theta\\
    &+\nu_1\wedge\omega^2-h^{-3/4}\rho_-\wedge d\theta.
\end{align*}
    Applying \eqref{eq: star_7} and \eqref{eq: rho^rho}, we get
    \begin{align*}
        \star_7d\psi=&\ h^{3/4}\star_6(dh^{-3/4}\wedge\rho_-)-\frac43\sigma_0\omega+\star_6(\pi_1\wedge\rho_-)+\sigma_2\\
        &\ -\left(2h^{-3/4}J\nu_1+h^{-3/2}\star_6(d\theta\wedge\rho_-)\right)\wedge\theta,
    \end{align*}
    where we used $\sigma_2=-\star_6(\sigma_2\wedge\omega)$. Using \eqref{eq: rho^omega} and \eqref{eq: rho^rho}, we get 
    \begin{align*}
        (\star_7d\psi)\wedge\varphi=&\ \left(dh^{-3/4}\wedge\rho_--\frac{4}{3h^{3/4}}\sigma_0\omega^2+h^{-3/4}\pi_1\wedge\rho_-+h^{-3/4}\sigma_2\wedge\omega\right.\\
        &\ \left.+\frac{1}{h^{3/2}}\star_6(d\theta\wedge\rho_-)\wedge\rho_++\frac{2}{h^{3/4}}J\nu_1\wedge\rho_+\right)\wedge\theta\\
        &\ +2\star_6J\pi_1+2h^{3/4}\star_6Jdh^{-3/4}.
    \end{align*}
If $\star_7d\psi\in\Omega^2_7$, then $\star_7d\psi\wedge\varphi=2d\psi$ is equivalent with:
   \begin{align}\label{eq: integrable 1-form}
    2\nu_1-h^{-3/4}J\star_6(d\theta\wedge\rho_-)=\pi_1+h^{3/4}dh^{-3/4}
   \end{align}
   and
   \begin{multline*}
       \left(dh^{-3/4}+h^{-3/4}\pi_1\right)\wedge\rho_-+\left(2h^{-3/4}J\nu_1+h^{-3/2}\star_6(d\theta\wedge\rho_-)\right)\wedge\rho_+-\frac{4}{3}\sigma_0h^{-3/4}\omega^2\\+h^{-3/4}\sigma_2\wedge\omega=2(dh^{-3/4}+h^{-3/4}\pi_1)\wedge\rho_--\frac{4}{3}\sigma_0h^{-3/4}\omega^2-2h^{-3/4}\sigma_2\wedge\omega.
   \end{multline*}
   Using \eqref{eq: rho^omega}, the last equation can be rewritten as
     \begin{equation}\label{eq: integrable 5-form}
         \left(2h^{-3/4}\nu_1-h^{-3/2}J\star_6(d\theta\wedge\rho_-)\right)\wedge\rho_-=(dh^{-3/4}+h^{-3/4}\pi_1)\wedge\rho_--h^{-3/4}\sigma_2\wedge\omega.
     \end{equation}
   Replacing \eqref{eq: integrable 1-form} into \eqref{eq: integrable 5-form}, we get $\sigma_2=0$. Finally, writing $d\theta=d\theta_1+d\theta_6+d\theta_8\in\Omega^2$ according to the decomposition \eqref{eq: su3 decomposition forms} and using $Jd\theta_6=-d\theta_6$, we obtain
   $$
     J\star_6(d\theta\wedge\rho_-)=\star_6J(d\theta_6\wedge\rho_-)=\star_6(d\theta_6\wedge\rho_+)=\star_6(d\theta\wedge\rho_+).
   $$
   Hence, we obtain \eqref{eq: integral_condition}. Now, for the torsion forms $\tau_0$ and $\tau_1$, we have
   \begin{align*}
       \tau_0=&\ \frac17\star_7\left(d\varphi\wedge\varphi\right)\\
       =&\ \frac17\star_7\left(h^{-3/2}\omega^2\wedge d\theta\wedge \theta+\frac23\pi_0\omega^3\wedge\theta+\pi_0\rho_+\wedge\rho_-\wedge\theta\right)\\
       =&\ \frac{1}{7h^{3/4}}\langle\omega, d\theta\rangle+\frac{8\pi_0}{7}.
   \end{align*}
  Similarly,  using \eqref{eq: star_7}, \eqref{eq: rho^rho} and \eqref{eq: rho_ and rho -}, we have
  \begin{align*}
      \tau_1=&\ \frac{1}{12}\star_7\left((\star_7d\psi)\wedge\psi\right)\\
      =&\ \frac{1}{12}\star_7\left(-\frac{4}{3}\sigma_0\omega^3+h^{-3/4}\star_6\left(\left(2\nu_1-\frac{1}{h^{3/4}}\star_6(d\theta\wedge\rho_+)\right)\wedge\rho_-\right)\wedge\rho_-\wedge\theta\right.\\
      &\left.-\frac{1}{h^{3/4}}\left(J\nu_1+\frac{1}{2h^{3/4}}\star_6(d\theta\wedge\rho_-)\right)\wedge\omega^2\wedge\theta\right)\\
      =&\ \frac12\nu_1-\frac{1}{4h^{3/4}}\star_6(d\theta\wedge\rho_-)-\frac{\sigma_0}{3h^{3/4}}\theta.
  \end{align*}
\end{proof}

\begin{cor}\label{cor: torsion_confor_coclosed_varphi_SU3}
    Let $M^1$ be an $\mathbb{S}^1$-bundle over $N^6$ with $\rG_2$-structure $\varphi$ given by \eqref{eq: S1-inv. G2-struct} and $\theta$ a connection $1$-form on $M$. The $\rG_2$-structure $\varphi$ is conformally coclosed if and only if $(h\omega,\rho_+,\theta)$ satisfies \eqref{eq: integral_condition} and there is $\phi\in C^\infty(N,\bR)$ such that
    \begin{align}
        d\phi=\frac12\nu_1-\frac{1}{4h^{3/4}}\star_6(F_\theta\wedge\rho_+), \quad \pi_1=2\nu_1-d(\log h^{-3/4}), \quad \sigma_0=0.
    \end{align}
\end{cor}

\begin{proof}
    Consider $\xi$ a vector field on $M$ dual to $\theta$, i.e.\ $\theta(\cdot)=h^{3/2}g_7(\xi,\cdot)$. If $\varphi$ is conformally coclosed, then there is $\phi\in C^\infty(M,\bR)$ such that $e^{-4\phi}\psi$ is closed and $\mathbb{S}^1$-invariant, then
    $$
    0=\cL_{\xi}\left(e^{-4\phi}\psi\right)=-d(e^{-4\phi}h^{-3/4}\rho_-).
    $$
    Using \eqref{eq: tau_0 SU3} with $\tau_1=d\phi$ implies $\sigma_0=0$.
\end{proof}

Given an $\SU(3)$-structure $(\omega,\rho_+)$, we recall the $3$-form $H_\omega\in \Omega^3(N)$ given by
\begin{equation}\label{eq: H_omega}
  H_\omega=-d^c\omega-\hat{N}
\end{equation}
where $d^c\omega=-Jd\omega$ and $\hat{N}$ is the totally skew-symmetric part of Nijenhuis tensor. In terms of the torsion forms \eqref{eq: torsion_SU3}, we have
\begin{equation}\label{eq: dc omega}
     d^c\omega=-Jd\omega=-\sigma_0\rho_-+\pi_0\rho_++\star_6(\nu_1\wedge\omega)-\star_6\nu_3, \quad \hat{N}=\frac{4}{3}\sigma_0\rho_--\frac{4}{3}\pi_0\rho_+.
\end{equation}
The next results relates the $3$-forms $H_\varphi$ and $H_\omega$.

\begin{lemma}\label{lm: dim_reduction H}
    Let $S^1\hookrightarrow M^7\to N^6$ as before and $\varphi$ an integrable $\rG_2$-structure on $M$ given by \eqref{eq: S1-inv. G2-struct} then
   \begin{align}\label{eq: H_varphi_su3}
    H_\varphi=&\ H_\omega+\frac{1}{3h^{3/4}}(\omega\lrcorner F_\theta)\rho_++\left(h^{3/4}(dh^{-3/4})+\frac{1}{h^{3/2}}(\star_6(F_\theta\wedge\rho_+)^\sharp\right)\lrcorner\frac{\omega^2}{2}\\ \nonumber
    &\ -\frac{1}{h^{3/4}}\left(\frac{1}{h^{3/4}}\left(\frac13(\omega\lrcorner F_\theta)\omega-F_\theta\right)+\pi_2\right)\wedge\theta.
\end{align} 
\end{lemma}

\begin{proof}
    Using Proposition \ref{prop: tau2=0 SU3}, we can compute $H_\varphi$ for the $\rG_2$-structure \eqref{eq: S1-inv. G2-struct}:
\begin{align*}
    H_\varphi=&\ \frac76\tau_0\varphi+4\star_7(\tau_1\wedge\varphi)-\star_7d\varphi\\
    =&\ \frac{1}{3h^{3/2}}\langle d\theta,\omega\rangle\omega\wedge\theta+\frac{1}{3h^{3/4}}\langle d\theta,\omega\rangle\rho_++\frac{4\pi_0}{3h^{3/4}}\omega\wedge\theta+\frac{4\pi_0}{3}\rho_+-\frac{4\sigma_0}{3}\rho_-\\
    &-\frac{1}{h^{3/2}}\star_6\left(\star_6(d\theta\wedge\rho_+)\wedge\rho_+\right)\wedge\theta+\frac{1}{h^{3/2}}\star_6\left(\star_6(d\theta\wedge\rho_+)\wedge\omega\right)+\frac{2}{h^{3/4}}\star_6(\nu_1\wedge\rho_+)\wedge\theta\\
    &-2\star_6(\nu_1\wedge\omega)+h^{3/4}(dh^{-3/4})\lrcorner\frac{\omega^2}{2}+\sigma_0\rho_--\pi_0\rho_++\star_6(\nu_1\wedge\omega)+\star_6\nu_3\\
    &-\frac{1}{h^{3/2}}\star_6(d\theta\wedge\omega)\wedge\theta-\frac{1}{h^{3/4}}\left(\frac{4}{3}\pi_0\omega+\star_6(\pi_1\wedge\rho_+)+\pi_2\right)\wedge\theta.
\end{align*}
Replacing $\pi_1$ by \eqref{eq: integral_condition} and the definition of $H_\omega$ \eqref{eq: H_omega} according to the torsion forms \eqref{eq: dc omega}, we get
\begin{align*}
    H_\varphi=&\ H_\omega+\frac{1}{3h^{3/4}}\langle d\theta,\omega\rangle\rho_++\left(h^{3/4}(dh^{-3/4})+\frac{1}{h^{3/2}}(\star_6(d\theta\wedge\rho_+)\right)^\sharp\lrcorner\frac{\omega^2}{2}\\
    &\ +\frac{1}{h^{3/4}}\left(\frac{1}{3h^{3/4}}\langle d\theta,\omega\rangle\omega-\frac{1}{h^{3/4}}\star_6(d\theta\wedge\omega)-\frac{1}{h^{3/4}}\star_6(\star_6(d\theta\wedge\rho_+)\wedge\rho_+)-\pi_2\right)\wedge\theta.
\end{align*}
Using the $\SU(3)$-decomposition of $\Omega^2$ as in \eqref{eq: su3 components forms}, we have
\begin{equation*}
    \star_6(d\theta\wedge\omega)=2d\theta_1+d\theta_6-d\theta_8, \quad \star_6(\star_6(d\theta\wedge\rho_+)\wedge\rho_+)=-2d\theta_6.
\end{equation*}
Notice that $d\theta_1=\frac13\langle d\theta,\omega\rangle\omega$, then
\begin{align*}
    H_\varphi=&\ H_\omega+\frac{1}{3h^{3/4}}\langle d\theta,\omega\rangle\rho_++\left(h^{3/4}(dh^{-3/4})^{\sharp}+\frac{1}{h^{3/2}}(\star_6(d\theta\wedge\rho_+)^\sharp\right)\lrcorner\frac{\omega^2}{2}\\
    &-\frac{1}{h^{3/4}}\left(\frac{1}{h^{3/4}}\left(d\theta_1-d\theta_6-d\theta_8\right)+\pi_2\right)\wedge\theta.
\end{align*}
\end{proof}

As a consequence of Lemma \ref{lm: dim_reduction H}, we have the following formula for $dH_\varphi$, in particular, we specialize this to $\theta$ an $\SU(3)$-instanton, $h$ a constant map and $N^6$ with totally skew-symmetric Nijenhuis tensor on $N^6$.

\begin{cor} \label{cor: dim_reduction dH}
The exterior derivative of \eqref{eq: H_varphi_su3} is
    \begin{align}\nonumber
    dH_\varphi=&\ dH_\omega+d\left(\frac{(\omega\lrcorner F_\theta)}{3h^{3/4}}\rho_+\right)+d\left(\left(h^{3/4}(dh^{-3/4})+\frac{1}{h^{3/2}}(\star_6(F_\theta\wedge\rho_+)^\sharp\right)\lrcorner\frac{\omega^2}{2}\right)\\ \label{eq: dH_varphi_su3}
    &-\frac{1}{h^{3/4}}\left(\frac{1}{h^{3/4}}\left(\frac13(\omega\lrcorner F_\theta)\omega-F_\theta\right)+\pi_2\right)\wedge F_\theta\\ \nonumber
    &-d\left(\frac{1}{h^{3/2}}\left(\frac13(\omega\lrcorner F_\theta)\omega-F_\theta\right)+\frac{1}{h^{3/4}}\pi_2\right)\wedge\theta.
\end{align} 
    In particular, if $d\theta\in \Omega^2_8$ then
    \begin{align*}
        dH_{\varphi}=&\ dH_\omega-d\left(\left(h^{3/4}dh^{-3/4} \right)^\sharp\lrcorner\frac{\omega^2}{2}\right)+\frac{1}{h^{3/4}}\left(\frac{1}{h^{3/4}}d\theta-\pi_2\right)\wedge d\theta\\
        &\ +\left(dh^{-3/4}\wedge\left(2h^{-3/4}d\theta-\pi_2\right)-h^{-3/4}d\pi_2\right)\wedge\theta.
    \end{align*}
    Moreover, if $h=c$ is constant and $\pi_2=0$ (i.e.\ the Nijenhuis tensor is totally skew-symmetric), then $H_\varphi$ is closed if and only if
    \begin{align*}
        dH_\omega+\frac{1}{c^{3/2}}F_\theta\wedge F_\theta=0.
    \end{align*}
\end{cor}

\subsection{Dimensional reduction of the heterotic \texorpdfstring{$\rG_2$}{G2}-flow}

In this section we consider the dimensional reduction of the generic heterotic $\rG_2$ flow \eqref{eq: mod_anomaly_flowCbeta}, under the ansatz \eqref{eq: S1-inv. G2-struct}.
For the $\rG_2$-structure \eqref{eq: S1-inv. G2-struct}, the generic heterotic $\rG_2$ flow \eqref{eq: mod_anomaly_flowCbeta} induces the following flow of $\SU(3)$-structures.  We also give a characterization of the fixed points.  A closely related description of the dimension reduced geometry of a $\rG_2$-structure with closed torsion was given recently in \cite{kennon2025canonical}.

\begin{prop}\label{prop:S1reducedflow}
    Let $\varphi$ be the $\rG_2$-structure \eqref{eq: S1-inv. G2-struct} and $\phi\in C^\infty(N,\bR)$. Then, the pair $(\varphi,\phi)$ satisfy the generic heterotic $\rG_2$ flow \eqref{eq: mod_anomaly_flowCbeta} if and only if $(\omega,\rho_-,\phi,h,\theta)$ satisfies
    \begin{gather*}
    \begin{split}
        \partial_t \left(e^{-4\phi}\frac{\omega^2}{2}\right)=&\ \gamma+\left(\frac{h^{3/4}}{2}\omega\lrcorner\lambda+\frac{1}{2}\rho_-\lrcorner\gamma\right)\wedge\rho_-, \qquad   \partial_t\left(e^{-4\phi}h^{-3/4}\rho_- \right)=\ \lambda\\
      \partial_t\theta 
    =&\frac{e^{4\phi}h^{3/2}}{2}\omega\lrcorner\lambda+\frac{e^{4\phi}h^{-3/4}}{2}\rho_-\lrcorner\gamma 
    \end{split}
    \end{gather*}
    where the forms
    $\lambda\in \Omega^3(N)$ and $\gamma\in \Omega^4(N)$ are given by
    \begin{align}
        \lambda =& \ d\left(\frac{1}{h^{3/2}}\left(\frac13(\omega\lrcorner F_\theta)\omega-F_\theta\right)+\frac{1}{h^{3/4}}\pi_2\right)+\frac12 Cd\left(\left(\frac{(\omega\lrcorner F_\theta)}{h^{3/4}}+4\pi_0\right)\frac{\omega}{h^{3/4}}\right)\\ \nonumber
        \gamma=& \ -dH_\omega-d\left(\frac{(\omega\lrcorner F_\theta)}{3h^{3/4}}\rho_+\right)-d\left(\left(h^{3/4}(dh^{-3/4})+\frac{1}{h^{3/2}}(\star_6(F_\theta\wedge\rho_+)^\sharp\right)\lrcorner\frac{\omega^2}{2}\right)\\ 
    &+\frac{1}{h^{3/4}}\left(\frac{1}{h^{3/4}}\left(\frac13(\omega\lrcorner F_\theta)\omega-F_\theta\right)+\pi_2\right)\wedge F_\theta\\ \nonumber
    &+\frac12Cd\left(\left(\frac{(\omega\lrcorner F_\theta)}{h^{3/4}}+4\pi_0\right)\rho_+\right)+\frac12\frac{C}{h^{3/4}}\left(\frac{(\omega\lrcorner F_\theta)}{h^{3/4}}+4\pi_0\right)\omega\wedge F_\theta.
    \end{align}
\end{prop}

\begin{proof}
Taking the time derivative of the conformal transformation of \eqref{eq: psi_S1-inv} 
\begin{align*}
    \frac{\partial}{\partial t}\left(e^{-4\phi}\psi\right)=&-4e^{-4\phi}\frac{\partial \phi}{\partial t}\left(\frac{\omega^2}{2}+h^{-3/4}\rho_-\wedge\theta\right)+e^{-4\phi}\frac{\partial}{\partial t}\left(\frac{\omega^2}{2}\right)+e^{-4\phi}\frac{\partial}{\partial t}\left(h^{-3/4}\rho_-\right)\wedge\theta\\
    &+e^{-4\phi}h^{-3/4}\rho_-\wedge\frac{\partial \theta}{\partial t}.
\end{align*}

Writing $-dH_\varphi+\frac74Cd\left(\tau_0\varphi\right)=\gamma+\lambda\wedge\theta$ for $\gamma\in\Omega^4(N)$ and $\lambda\in\Omega^3$,  the evolution of the $4$-form \eqref{eq: mod_anomaly_flowCbeta} becomes
\begin{align}\label{eq: 4-form su3 evolution}
    -4\frac{\partial \phi}{\partial t}\frac{\omega^2}{2}+\frac{\partial}{\partial t}\frac{\omega^2}{2}+h^{-3/4}\rho_-\wedge\frac{\partial \theta}{\partial t}=&\ e^{4\phi}\gamma\\ \label{eq: 3-form su3 evolution}
    \frac{\partial \phi}{\partial t}h^{-3/4}\rho_-+\frac{\partial}{\partial t}\left(h^{-3/4}\rho_-\right)=&\ e^{4\phi}\lambda
\end{align}

Taking $\star_6$ in \eqref{eq: 4-form su3 evolution} and wedging with $h^{-3/4}\rho_-$, by \eqref{eq: rho_ and rho -} we have
$$
  \star_6\left(\frac{\partial}{\partial t}\frac{\omega^2}{2}\right)\wedge(h^{-3/4}\rho_-)+h^{-3/2}\star_6\frac{\partial \theta}{\partial t}=e^{4\phi}\star_6(\gamma)\wedge(h^{-3/4}\rho_-)
$$
From \eqref{eq: ddt *}, we have
$$
  \star_6\left(\frac{\partial}{\partial t}\frac{\omega^2}{2}\right)=\frac{\partial}{\partial t}\omega+\star_6(S\diamond\omega^2)-\tr(S)\omega
$$
for a symmetric tensor on $N$. Since $\star_6(S\diamond\omega^2)-\tr(S)\omega\in\Omega^2_1\oplus\Omega^2_8$ and $\omega\wedge\rho_-=0$, then
\begin{align*}
    \frac{\partial\theta}{\partial t}=&\ \frac{h^{3/2}}{2}\star_6\left(\omega\wedge\frac{\partial}{\partial t}(h^{-3/4}\rho_-)\right)+\frac{e^{4\phi}h^{-3/4}}{2}\star_6\left(\star_6\gamma\wedge\rho_-\right)\\
    =&\ \frac{e^{4\phi}h^{3/2}}{2}\star_6\left(\omega\wedge\lambda\right)+\frac{e^{4\phi}h^{-3/4}}{2}\star_6\left(\star_6\gamma\wedge\rho_-\right)
\end{align*}
Finally, writing $\omega$ and $\rho_-$ in coordinates,  and using the identity $\star_6(\alpha\wedge\sigma)=(-1)^k\alpha^\sharp\lrcorner\star_6\sigma$, we obtain $\omega\lrcorner\lambda=\star_6\left(\omega\wedge\lambda\right)$ and $\rho_-\lrcorner\gamma=\star_6\left(\star_6\gamma\wedge\rho_-\right)$
\end{proof}

We specialize the previous proposition to the Riemannian product $M=N\times \mathbb{S}^1$ and flat connection, i.e.\ $h=1$ and $d\theta=0$.

\begin{cor}
    Let $\varphi$ be the $\rG_2$-structure \eqref{eq: S1-inv. G2-struct}, with $h=1$, $d\theta=0$ and $\phi\in C^\infty(N,\bR)$. Then, the pair $(\varphi,\phi)$ satisfy the generic heterotic $\rG_2$ flow \eqref{eq: mod_anomaly_flowCbeta} if and only if $(\omega,\rho_-,\phi)$ satisfies
    \begin{align}\nonumber
      \partial_t \left(e^{-4\phi}\frac{\omega^2}{2}\right)=& \ -dH_\omega+2Cd\left(\pi_0\rho_+\right),\\ \label{eq: 6-dim anomaly flow}
        \partial_t\left(e^{-4\phi}\rho_- \right)=& \ d\left(\omega\lrcorner d\rho_+\right)+2\left(C-\frac{2}{3}\right)d(\pi_0\omega)-8d(\nabla\phi\lrcorner\rho_-).\\ \nonumber
        \partial_t\phi=& \ e^{4\phi}\left(\Delta\phi+(4-\gamma)|d\phi|^2+\frac14|H_\omega|^2+\frac14|\pi_2|^2-\left(\frac{28}{63}+\sigma\right)\pi_0^2\right).
    \end{align}
    In particular, the critical points of \eqref{eq: 6-dim anomaly flow} are $\SU(3)$-structures $(\omega,\rho_+)$ with Lee form $\pi_1=\nu_1$ exact, $\sigma_2=0$, $\pi_0=\sigma_0=0$ and
    $$
      dH_\omega=0, \qquad d\pi_2=0.
    $$ 
\end{cor}

\begin{proof}
    The evolution equations  \eqref{eq: 6-dim anomaly flow} follow from Corollary \ref{cor: dim_reduction dH}. Finally, notice that 
\begin{equation*}
    \omega\lrcorner(d\rho_+)=\frac{4}{3}\pi_0\omega-8\nabla\phi\lrcorner\rho_-+\pi_2
\end{equation*}
and taking the exterior derivative, we have
\begin{equation*}
    \partial_t\left(e^{-4\phi}\rho_-\right)=d\Lambda_\omega(d\rho_+)+2\left(C-\frac{2}{3}\right)d(\pi_0\omega)-8d(\nabla\phi\lrcorner\rho_-).
\end{equation*}
Moreover, from Proposition \ref{prop:weakG2} and Corollary \ref{cor: torsion_confor_coclosed_varphi_SU3}, we obtain that the critical points of \eqref{eq: 6-dim anomaly flow} satisfy 
    $$
     \pi_1=2\nu_1=4d\phi_6, \quad \sigma_2=0, \quad \pi_0=0, \quad dH_\omega=0, \quad d\pi_2=0, \quad d\sigma_0=0.
    $$
\end{proof}

\begin{rmk}\label{rem:interpolation}
Formula \eqref{eq: 6-dim anomaly flow} shows an interesting link between the heterotic $\rG_2$ flow and natural flows for $\SU(3)$-structures, previously introduced in the literature. For instance, assuming that the initial condition satisfies $d \rho_+ = d \rho_- = 0$, the first line in \eqref{eq: 6-dim anomaly flow} recovers the \emph{anomaly flow} of Phong, Picard, Zhang \cite{PPZ1,PPZ2} for conformally coclosed integrable $\operatorname{SU}(3)$-structures, provided we take $e^{-4\phi}$ to be the norm of the corresponding holomorphic volume form. On the other hand, if we assume instead that $d\omega = 0$, the second line in \eqref{eq: 6-dim anomaly flow} recovers a gauge-fixed version of the \emph{Type IIA flow} \cite{IIAflow}. Thus, formally at least, the previous coupled system of evolution equations provides a natural \emph{interpolation} between the aforementioned flows. This suggests natural modifications of the anomaly and Type IIA flows, where the dilaton has an independent evolution equation, which may prevent finite-time singularities along the flow (cf. \cite[Section 5.4]{Picard2024}).  One may speculate that there is a link to the $T$-duality and non-K\"ahler mirror symmetry (cf.\ \cite{Fei:2023}).
\end{rmk}


\providecommand{\bysame}{\leavevmode\hbox to3em{\hrulefill}\thinspace}
\providecommand{\MR}{\relax\ifhmode\unskip\space\fi MR }
\providecommand{\MRhref}[2]{%
  \href{http://www.ams.org/mathscinet-getitem?mr=#1}{#2}
}
\providecommand{\href}[2]{#2}

\end{document}